\documentclass[a4paper, reqno]{amsart}
\usepackage{amsmath}
\usepackage{mathtools}
\usepackage{amssymb}
\usepackage{enumerate}
\usepackage{amsthm}
\usepackage{mathrsfs}
\usepackage{cancel}
\usepackage{array}
\usepackage{color}
\usepackage{geometry}

\DeclareMathAlphabet{\mathpzc}{OT1}{pzc}{m}{it}
\makeatletter
\newcommand*{\rom}[1]{\expandafter\@slowromancap\romannumeral #1@}
\makeatother
\makeatletter

\newcommand{\C}{\mathbb{C}}

\newcommand{\R}{\mathbb{R}}

\newcommand{\N}{\mathbb{N}}

\theoremstyle{definition}
\newtheorem*{defn}{Definition}
\theoremstyle{plain}
\newtheorem*{thm*}{Theorem}
\theoremstyle{plain}
\newtheorem{thm}{Theorem}[section]
\theoremstyle{plain}
\newtheorem{lem}[thm]{Lemma}
\theoremstyle{plain}

\theoremstyle{plain}
\newtheorem{prop}[thm]{Proposition}
\theoremstyle{remark}
\newtheorem{rmk}{Remark}

\title[]{Solitary waves for the nonlinear Schr\"odinger-Poisson system with positron-electron interaction}

\author[S. Jin]{Sangdon Jin}
\address{Department of mathematics, Chung-Ang University, Seoul 06974, Korea}
\email{sdjin@cau.ac.kr}
\author[J. Seok]{Jinmyoung Seok}
\address{Department of Mathematics, Kyonggi University, Suwon 16227, Republic of Korea}
\email{jmseok@kgu.ac.kr}

\begin{document}

\begin{abstract}
In this paper, we study the existence of positive solutions to the nonlinear elliptic system
\begin{equation*}  
\begin{cases} -\Delta u_1+  u_1+(\mu_{11}\phi_{u_1}-\mu_{12}\phi_{u_2})u_1=\frac{1}{2\pi}\int_0^{2\pi} |u_1+e^{i\theta}u_2|^{p-1}(u_1+e^{i\theta}u_2)d\theta \ \ \mbox{ in }\ \  \R^3,
\\  -\Delta u_2+ u_2+(\mu_{22}\phi_{u_2}-\mu_{12}\phi_{u_1})u_2=\frac{1}{2\pi}\int_0^{2\pi} |u_2+e^{i\theta}u_1|^{p-1}(u_2+e^{i\theta}u_1) d\theta \ \ \mbox{ in }\ \  \R^3, \end{cases}
\end{equation*}
which is derived from taking the nonrelativistic limit of the nonlinear Maxwell-Klein-Gordon equations under the decomposition of waves functions into positron and electron parts. 
We characterize the existence and nonexistence of positive vector solutions, depending on parameters $p$ and $\mu_{ij}$.
\end{abstract}

\subjclass[2010]{35J50, 35J61, 35Q40}
\keywords{semilinear elliptic system; Maxwell-Klein-Gordon; Schr\"odinger-Poisson; variational method}

\maketitle

\section{ Introduction}
This paper is devoted to classify the structure of standing waves for the following system of Schr\"odinger-Poisson type,  
\begin{equation}\label{mains}
\left\{\begin{aligned}
&2 i\dot{v}_+-\Delta v_++(\mu_{11}\phi_{v_+}-\mu_{12}\phi_{v_-})v_+
-\frac{1}{2\pi}\int_0^{2\pi}g(v_++e^{i\theta}\overline{v}_-)d\theta=0,\\
& 2i\dot{v}_--\Delta v_-+(\mu_{22}\phi_{v_-}-\mu_{12}\phi_{v_+})v_-
-\frac{1}{2\pi}\int_0^{2\pi}g(v_-+e^{i\theta}\overline{v}_+)d\theta=0,  
\end{aligned}\right.
\end{equation}
where $v_\pm(t,x):\R^{1+3}\rightarrow \C$, $\phi_u=(4\pi|\cdot|)^{-1} \ast |u|^2$, $\mu_{ij} >0$ for  $i,j=1,2$  and $g(v):\C \rightarrow \C $ satisfying $g(e^{is} v)=e^{is}g(v)$ for all $s\in \R$.  
This system arises as a limit system of the nonlinear Maxwell-Klein-Gordon equations when we consider the nonrelativistic limit, taking the speed of light to infinity. 
The nonlinear Maxwell-Klein-Gordon equations are written as follows: 
\begin{equation}\tag{NMKG}
\left\{\begin{aligned}
& D_\alpha D^{\alpha}\psi = (mc)^2\psi -g(\psi), \\
& \partial^{\beta}F_{\alpha\beta} = \frac{1}{c}\text{Im}(\psi\overline{D_\alpha\psi}),
\end{aligned}\right.  
\end{equation}
where $\psi(t,x):\R^{1+3}\rightarrow \C$ is a wave function, $g :\C\rightarrow \C$ is a nonlinear potential satisfying $g(e^{is}v)=e^{is}g(v)$ for all $s\in \R$, 
$D_\alpha =\partial_\alpha +\frac{i}{c}A_\alpha, \alpha = 0, 1, 2, 3$ is the covariant derivative and
$F_{\alpha\beta} =\partial_\alpha A_\beta-\partial_{\beta}A_{\alpha}$ is the curvature tensor. 
Here, we write $\partial_0 = \frac{\partial}{c\partial t}$, $\partial_j = \frac{\partial}{\partial x_{j}}, j = 1, 2, 3$.
Indices are raised under the Minkowski metric $g_{\alpha\beta} = \text{diag}(-1, 1, 1, 1)$, i.e., $X^{\alpha} = g_{\alpha\beta}X_{\beta}$.
The positron part $v_+$ and electron part $v_-$ of a solution $\psi$ to (NMKG) are given by
\[
v_+ \coloneqq e^{-ic^2t}\frac12\left(\psi-i\langle\nabla\rangle_c^{-1}D^+_0\psi\right),\quad
v_- \coloneqq e^{-ic^2t}\frac12\left(\bar\psi-i\langle\nabla\rangle_c^{-1}D^-_0\bar\psi\right),
\]
where $\langle \nabla\rangle_c$ denotes the Fourier multiplier with the symbol $\sqrt{|\cdot|^2+c^2}$ and
$D_\alpha^{\pm} = \partial_\alpha \pm \frac{i}{c}A_\alpha$.
Some formal computation (for example, see \cite{MN}) shows that by taking the nonrelativistic limit $c\to\infty$,
$(v_+,v_-)$ approaches to a solution of the system \eqref{mains} with $\mu_{ij} = 2$.
In other words, a solution of (NMKG) can be approximated as the two different modes of oscillations, namely  $\psi=e^{ic^2t}v_++e^{-ic^2t}\overline{v_-}+o_c(1)$ as $c\rightarrow \infty$, where $(v_+,v_-)$ is a solution to the system \eqref{mains}. The integration $\frac{1}{2\pi}\int_0^{2\pi}g(v_\pm+e^{i\theta}\overline{v}_\mp)d\theta$ in \eqref{mains} describes the resonances and if we take the cubic nonlinearity $g(v)=|v|^2v$, then it simply becomes $(|v_{\pm}|^2+2|v_\mp|^2)v_\pm$.

In \cite{MN, MN1}, Masmoudi-Nakanishi rigorously justified the convergence of nonrelativistic limit in the space $C[0,T], H^1)$ either the Maxwell gauge terms are not involved, that is, $A_\alpha=0$ for $\alpha = 0, 1, 2, 3$ or
the nonlinear potential term $g$ is absent. 
In these cases, the system \eqref{mains} reduces to respectively the systems
\begin{equation}\label{coupledNLS}
2i\dot{v}_{\pm}-\Delta v_{\pm}-\frac{1}{2\pi}\int_0^{2\pi}g(v_\pm+e^{i\theta}\overline{v_\mp})d\theta=0
\end{equation}
or
\begin{equation}\label{htr}
\left\{\begin{aligned}
&2 i\dot{v}_+-\Delta v_++(\mu_{11}\phi_{v_+}-\mu_{12}\phi_{v_-})v_+ =0,\\
& 2i\dot{v}_--\Delta v_-+(\mu_{22}\phi_{v_-}-\mu_{12}\phi_{v_+})v_- =0. 
\end{aligned}\right.
\end{equation}

We note that the system \eqref{mains} serves as a physically meaningful generalization of several well-studied semilinear PDEs.
For example, when either $v_+$ or $v_-$ vanishes, the system \eqref{mains} is reduced to the so-called Schr\"odinger-Poisson equation
\[
2i\dot{v} -\Delta v + \mu \phi_{v} v - g(v )=0,
\]
which is extensively studied during past two decades. We refer to the literatures \cite{C, CW, DM1, DM, MRS, R}
for the existence and qualitative properties of positive standing waves. 
The system \eqref{coupledNLS} with $g(u)=|u|^2u$ is called the coupled system of cubic nonlinear Schr\"odinger equations and  appeared in theory of Bose–Einstein condensates and nonlinear optics (see \cite{AA,AC, BDW, LW, MCSS, MS, PW, RCF,S1, TBDC,WY}  and references therein).  
Moreover, if we take $v_+ = v_-$ and $\mu_{11} = \mu_{22}$ in the system \eqref{htr},  we obtain the nonlinear Hartree equation
$$
2 i\dot{v} -\Delta v + \mu \phi_{v} v =0,
$$
which arises in the self gravitational collapse of quantum mechanical system and physics of laser beams.
We refer to \cite{L, Li, Le, MZ} and references therein for the study of ground states and standing waves. 
In relating with the system \eqref{htr}, the existence of positive standing waves in \cite{WS} is studied when the parameters $\mu_{11}$ and $\mu_{22}$ are negative. 

The main novelty of this paper therefore is to investigate the existence of standing waves having nontrivial components $v_+,\, v_-$ to the full system \eqref{mains} with positive coupling parameters $\mu_{ij} > 0$. 
We will see that the nontrivial standing waves of the aforementioned reduced systems play a significant role in analyzing structures of standing waves to \eqref{mains}.

Now, we insert the standing wave ansatz
\[
v_+(x,t)=u_1(x)e^{i\frac{\lambda}{2} t} \mbox{ and } v_-(x,t)=u_2(x)e^{i\frac{\lambda}{2} t}
\]
into the system \eqref{mains} to obtain 
\begin{equation}\label{gmee} 
\begin{cases} -\Delta u_1+\lambda u_1+(\mu_{11}\phi_{u_1}-\mu_{12}\phi_{u_2})u_1=\frac{1}{2\pi}\int_0^{2\pi} g(u_1+e^{i\theta}u_2)d\theta \ \ \mbox{ in }\ \  \R^3,
\\  -\Delta u_2+\lambda u_2+(\mu_{22}\phi_{u_2}-\mu_{12}\phi_{u_1})u_2=\frac{1}{2\pi}\int_0^{2\pi} g(u_2+e^{i\theta}u_1)d\theta \ \ \mbox{ in }\ \  \R^3, \end{cases}
\end{equation}
where $u_1, u_2:\R^3\rightarrow \R$, $\lambda, \mu_{ij} >0$ for  $i,j=1,2$.
We denote by $\vec{u}$ a pair of functions $(u_1,\, u_2)$. 
Here we define some concepts of triviality and positiveness of a vector function $\vec{u}$.
\begin{defn} 
\begin{enumerate}[$*$] 
A vector function $\vec{u} = (u_1,u_2)$ is said to be \smallskip
\item nontrivial if either $u_1 \neq 0$ or $u_2 \neq 0$; 
\item semi-trivial if it is nontrivial but either $u_1 = 0$ or $u_2 = 0$;
\item vectorial if both of $u_1$ and $u_2$ are not zero;
\item nonnegative if $u_1 \geq 0$ and $u_2 \geq 0$;
\item  positive if $u_1 > 0$ and $u_2 > 0$.
\end{enumerate}
\end{defn}

We denote by $H^1 = H^1(\R^3)$ the completion of $C^\infty_c(\R^3)$, the space of real valued smooth function with compact support, with respect to the standard Sobolev norm
\[
\|u \|_{\lambda}=\Big(\int_{\R^3}|\nabla u|^2+\lambda u^2dx\Big)^\frac12.
\]
The notation $H_r^1$ denotes the space of radially symmetric functions in $H^1$. 
We also define ${\bf H} \coloneqq H_r^1\times H_r^1$ equipped with the norm 
\[
\|\vec{u}\|_{\bf{H}}^2= \|u_1 \|_{\lambda}^2+\|u_2 \|_{\lambda}^2. 
\]

We now are ready to state the main results. Our first result deals with the case $g = 0$, in which the system \eqref{gmee} is reduced to
\begin{equation}\label{bme}
\begin{cases} 
-\Delta u_1+\lambda u_1+(\mu_{11}\phi_{u_1}-\mu_{12}\phi_{u_2})u_1=0 \ \ \mbox{ in }\ \  \R^3,\\  
-\Delta u_2+\lambda u_2+(\mu_{22}\phi_{u_2}-\mu_{12}\phi_{u_1})u_2=0\ \ \mbox{ in }\ \  \R^3.
\end{cases}
\end{equation}
In this case we can provide a satisfactory picture for the solution structure of \eqref{bme}.
Interestingly, it turns out that the sign of determinant of the matrix $(\mu_{ij})$ completely determines the existence of nontrivial solutions to \eqref{bme}.

\begin{thm}\label{btm} 
Let  $\lambda,  \mu_{ij}  >0$   for  $i, j=1,2$. If $\mu_{11}\mu_{22}\ge \mu_{12}^2$, then \eqref{bme} has only the trivial solution in $H^1(\R^3)\times H^1(\R^3)$. 
If $\mu_{11}\mu_{22}< \mu_{12}^2$, 
The system \eqref{bme} admits a positive vector solution $\vec{u} \in \mathbf{H}$, which is unique in the class of nontrivial nonnegative functions in {\bf H}.
Actually, it is determined as $\vec{u}=(V,a_0V)$, where $a_0=\sqrt{\frac{\mu_{11}+\mu_{12}}{\mu_{22}+\mu_{12}}}$ and 
$V$ is a unique positive radial solution of the equation
\begin{equation}\label{bme1}
-\Delta u+\lambda u+ \frac{\mu_{11}\mu_{22}-\mu_{12}^2}{\mu_{22}+\mu_{12}} \phi_{u} u=0.
\end{equation}
\end{thm}
\begin{rmk}\em
It seems interesting to compare the system \eqref{bme} with the corresponding single equation, so-called  nonlinear Hartree equation
\[
-\Delta u+\lambda u +\mu \phi_uu = 0.
\]
It is proved in \cite{L, Le} that it admit a unique positive radial solution when $\mu < 0$. (Note that this makes the statement of Theorem \ref{btm} valid.) By multiplying by $u$ and integrating, it is also easy to see it has only trivial solution if $\mu \geq 0$. 
Here, we point out that the sign of $\text{det}(\mu_{ij})$ of the system \eqref{bme} plays the same role with the sign of $\mu$.

\end{rmk}
\begin{rmk}\label{rmk-no-semi}\em
It is  immediate to show that the system \eqref{bme} does not have any semi-trivial solution. Suppose, for example, $u_2 = 0$ for a solution $\vec{u}$
so that \eqref{bme} reduces to $-\Delta u_1+\lambda u_1+\mu_{11}\phi_{u_1}u_1=0$.
Then as mentioned in the previous remark, we see $u_1 = 0$. 
\end{rmk}

We next turn our attention to the case $g \neq 0$. In this paper,  we only focus on the standard power function $g(v)=|v|^{p-1}v$.
Then we can write the system \eqref{gmee} as
\begin{equation}\label{gme}
\begin{cases} -\Delta u_1+\lambda u_1+(\mu_{11}\phi_{u_1}-\mu_{12}\phi_{u_2})u_1=\frac{1}{2\pi}\int_0^{2\pi} |u_1+e^{i\theta}u_2|^{p-1}(u_1+e^{i\theta}u_2)d\theta \ \ \mbox{ in }\ \  \R^3,
\\  -\Delta u_2+\lambda u_2+(\mu_{22}\phi_{u_2}-\mu_{12}\phi_{u_1})u_2=\frac{1}{2\pi}\int_0^{2\pi} |u_2+e^{i\theta}u_1|^{p-1}(u_2+e^{i\theta}u_1) d\theta \ \ \mbox{ in }\ \  \R^3. \end{cases}
\end{equation}
In this case, one can compare the system \eqref{gme}  with 
\begin{equation}\label{NSP}
 -\Delta u+ \lambda  u + \mu\phi_{u} u= |u|^{p-1} u \mbox{ in } \R^3,
\end{equation}
which is called the nonlinear Schr\"odinger-Poisson equation.
It is proved in \cite{DM} that when $p \in (0, 1] \cup [5,\infty)$, there exists no nontrivial solution to \eqref{NSP}. 
In the interesting paper \cite{R}, Ruiz investigated the existence of positive radial solutions to \eqref{NSP} with the subcritical range of $p \in (1,5)$.
He shows that while a positive radial solution exists when $p \in (2, 5)$ for each $\lambda, \mu > 0$, the situation is changed when we consider the case $p \in (1,2]$, where
there exists no nontrivial solution for large $\mu > 0$ and there exist  two positive radial solutions for small $\mu > 0$.

It is therefore naturally expected that  the system \eqref{gme}  admits only trivial solution if the range of the exponent $p > 0$ is either sub-linear or $H^1$ super-critical. The second result of this paper confirms this whenever $\det{(\mu_{ij})} \geq 0$. 
\begin{thm}\label{th0}
Let $\lambda, \mu_{ij} >0$ for  $i,j=1,2$ satisfying $\mu_{11}\mu_{22}\ge \mu_{12}^2$.  If  $p\in(0,1]\cup [5,\infty)$, then \eqref{gme} has only the trivial solution in $H^1(\R^3)\times H^1(\R^3)$.
\end{thm}
\begin{rmk}\em
The condition $\mu_{11}\mu_{22}\ge \mu_{12}^2$ is not a technical matter. Indeed, we note that, when $p=1$, \eqref{gme} becomes the following systems
\begin{equation}\label{endc}
\begin{cases} -\Delta u_1+(\lambda-1) u_1+(\mu_{11}\phi_{u_1}-\mu_{12}\phi_{u_2})u_1=0 \ \ \mbox{ in }\ \  \R^3,
\\  -\Delta u_2+(\lambda-1)u_2+(\mu_{22}\phi_{u_2}-\mu_{12}\phi_{u_1})u_2=0 \ \ \mbox{ in }\ \  \R^3.\end{cases}
\end{equation}
Then, by Theorem \ref{btm}, we deduce that if $\lambda >1$ and $\mu_{ij}>0$ with $\mu_{11}\mu_{22}<\mu_{12}^2$, then  \eqref{endc} has a positive solution. 
\end{rmk}

In the subsequent theorems, we investigate the structure of ground states to \eqref{gme} in the super-linear and sub-critical range of $p \in (1, 5)$ to search for positive vector solutions.
Observe that the system \eqref{gme} is the Euler-Lagrange system of the action functional
\begin{align*}
I(\vec{u})&=\frac12\int_{\R^3}|\nabla u_1|^2+|\nabla u_2|^2+\lambda u_1^2+\lambda u_2^2dx+\frac14\int_{\R^3} \mu_{11}u_1^2 \phi_{u_1} + \mu_{22}u_2^2 \phi_{u_2}-2\mu_{12} u_1^2 \phi_{u_2} dx\\
&\qquad - \frac{1}{p+1}\frac{1}{2\pi} \int_{\R^3}\int_0^{2\pi} |u_1+e^{i\theta}u_2|^{p+1}d\theta dx.
\end{align*} 
We say that a solution $\vec{u} \in \mathbf{H}$ to \eqref{gme} is a (radial) ground state if $\vec{u}$ is nontrivial and  minimizes the value of $I$ among all nontrivial solutions to \eqref{gme} in $\mathbf{H}$.

Similarly with the results on the nonlinear Schr\"odinger-Poisson equations, it turns out that the solution structure of \eqref{gme} changes at the borderline $p = 2$. 
The next theorem is concerned with the range of $p \in (2,5)$.
\begin{thm}\label{th3}
Let $2<p<5$ and $\lambda, \mu_{ij} >0$ for  $i,j=1,2$.
Then there exists a  non-negative ground state solution $\vec{w} $ of \eqref{gme}. 
If either $2<p<5$ and $\mu_{11}=\mu_{22}$; or $\frac13 (-2 + \sqrt{73})\le p<5$, then $\vec{w}$ is positive. 
Moreover, if  $2<p<5$ and $\mu_{11}=\mu_{22}$, then $\vec{w}$ must be of the form $\vec{w} =(U ,U )$, where $U$ is a positive ground state solution of 
\begin{equation}\label{sie}
 -\Delta u+ \lambda  u + ( {\mu_{11} - \mu_{12}}) \phi_{u} u= \frac{c_p}{2}|u|^{p-1} u \mbox{ in } \R^3
\end{equation}
with $c_p = \frac{2^{\frac{p-1}{2}}}{ \pi}\int_0^{2\pi} (1+\cos \theta)^{\frac{p+1}{2}}d\theta$.
\end{thm}
\begin{rmk}\em
We note that the number $\frac13 (-2 + \sqrt{73})\cong 2.18133$ is slightly larger than 2.
We believe that this number appears in just a technical matter and any ground state is actually positive in the whole range of $2< p <5$ without the assumption $\mu_{11} = \mu_{22}$.
\end{rmk}


Now, we turn to the remaining case $p \in (1, 2]$. 
It turns out that the solution structure varies when we change the sign of $\det{(\mu_{ij})}$.
We first construct a positive solution of \eqref{gme} when $\mu_{12}>0$ is taken to be large, i.e., $\det(\mu_{ij}) < 0$ and $|\det(\mu_{ij})|$ is large. 
The solution is obtained as a perturbation of $(V,V)$, where $V$ is the solution of a nonlinear Hartree equation
\[
-\Delta V+\lambda V-\phi_V V=0 \mbox{ in } \R^3.
\]
 
\begin{thm}\label{th5}
Let $1<p\le2$ and $\lambda, \mu_{11},\, \mu_{22} >0$ be fixed. Then there exists a constant $\mu_0 > 0$ such that  if $\mu_{12} > \mu_0$, the system \eqref{gme} admits a positive solution. 
\end{thm}


The next two theorems cover the case $\det{(\mu_{ij})} > 0$.
We will see that, unlike the case $\det{(\mu_{ij}) < 0}$, the system \eqref{gme} does not admit any nontrivial solution when $\det(\mu_{ij})$ is taken to be large.  
\begin{thm} \label{th31}
Let $1<p\le2$ and $\mu_{ij} >0$ for  $i,j=1,2$. Assume that $\lambda\ge2$, $\mu_{11}>4$ and $(\mu_{11}-4)(\mu_{22}-4)>\mu_{12}^2.$
 Then
\eqref{gme} has only the trivial solution in $H^1(\R^3)\times H^1(\R^3)$.
\end{thm}

Our final result deals with the case that $\det{(\mu_j)} > 0$ and $\mu_{ij}$ is small. 
We construct two positive solutions to \eqref{gme} such that one has positive energy and the other one is an energy minimizer which has negative energy.   
This corresponds to the work by Ruiz \cite{R} in which the same kind of positive solutions for \eqref{sie} are found. 
\begin{thm}\label{th4}
Let $1<p<2$ and $\lambda, \mu_{ij} >0$ for  $i,j=1,2$. If  $\mu_{11}\mu_{22}-\mu_{12}^2>0$  and $\mu_{ij}$ is small enough, where $i,j=1,2$, then  there are at least two different   positive radial solutions $\vec{u}_1$, $\vec{u}_2$ of  \eqref{gme}, where $\vec{u}_1$ is a positive minimizer of  \eqref{gme} with negative energy, and $\vec{u}_2$ is a positive solution  of \eqref{gme} with positive energy. Moreover, if we assume $\mu_{11}\mu_{22}-\mu_{12}^2>0$, $\mu_{11}=\mu_{22}$   and  $\mu_{ij}>0$ is small enough, where $i,j=1,2$, then  $\vec{u}_1$ and  $\vec{u}_2$ have the form $\vec{u}_1=(U_1,U_1)$  and  $\vec{u}_2=(U_2,U_2)$, where $U_1$ is   a positive minimizer of \eqref{sie} with negative energy, and  $U_2$ is a positive solution  of \eqref{sie} with  positive energy.
\end{thm}
\begin{rmk}\em
Theorem \ref{th5}  and \ref{th4} show that the value of $\det(\mu_{ij})$ affects the solution structure of \eqref{gme} in the case $p \in (1,2]$.
If $\det(\mu_{ij})$ is largely negative, then there exists a positive solution bifurcating from the ground state of the nonlinear Hartree equation. 
If $\det(\mu_{ij})$ is positive and the coefficient are small, then there exists two positive solution bifurcating from two positive solution of the nonlinear Schr\"odinger-Poisson equations. 
\end{rmk}


The rest of paper is organized as follows. In Section \ref{bmepr}, we study the system \eqref{gmee} with $g=0$ and prove Theorem \ref{btm}. Section \ref{strucsect} is devoted to provide classification results of positive solutions to \eqref{gme}, and some properties of the semi-trivial solutions to \eqref{gme} such as the energy levels or Morse indices of them. In Section \ref{nonlin}, we study the existence of positive solutions to \eqref{gme} for $2<p<5$, and prove Theorem \ref{th3}. In Section \ref{negadet}, we construct positive solutions to \eqref{gme} for $1<p\le 2$ and $\det(\mu_{ij})<0$, and prove Theorem \ref{th5}. In Section \ref{posidet}, we present the multiple existence and non-existence results to \eqref{gme}  for $1<p<2$ and $\det(\mu_{ij})>0$, and give the proofs of Theorem \ref{th31} and Theorem \ref{th4}.
In Appendix \ref{append1}, we prove  Pohozaev identity for \eqref{gme}, and give non-existence results to \eqref{gme} when $p\in (0,1]\cup [5,\infty)$, which provide the proof of Theorem \ref{th0}. Finally, in Appendix \ref{regularity}, we discuss some regularities of the energy functional for \eqref{gme}.

\section{Classification of solutions in zero potential case} \label{bmepr}
In this section, we prove Theorem \ref{btm}, which gives a complete classification of nonnegative solutions for \eqref{bme}.
Throughout this section, we assume  $\lambda,  \mu_{ij}  >0$ for $i,j=1,2$ and $\vec{u}=(u_1,u_2)\in  H^1(\R^3)\times H^1(\R^3)$ always denotes an arbitrary solution of \eqref{bme}.

\subsection{Triviality of $\vec{u}$ when $\det(\mu_{ij}) \geq 0$}
We first deal with the case $\det(\mu_{ij}) \geq 0$, i.e., $\mu_{11}\mu_{22}\ge \mu_{12}^2$. 
Our claim is that $\vec{u}$ is trivial.
Multiplying the first equation of \eqref{bme} by $u_1$ and the second equation by $u_2$ respectively, and integrating, we see that 
\begin{equation}\label{ch5}
\begin{aligned}
0&=\int_{\R^3}\Big(-\Delta u_1+\lambda u_1+(\mu_{11}\phi_{u_1}-\mu_{12}\phi_{u_2})u_1\Big)u_1dx =\int_{\R^3}|\nabla u_1|^2+\lambda u_1^2 +\mu_{11}\phi_{u_1}u_1^2-\mu_{12}\phi_{u_2}u_1^2dx\\
&=\int_{\R^3}|\nabla u_1|^2+\lambda u_1^2 +\mu_{11}\phi_{u_1}u_1^2-\mu_{12}\phi_{u_1}u_2^2dx
\end{aligned}
\end{equation}
and
\begin{equation}\label{ch6}
0=\int_{\R^3}\Big( -\Delta u_2+\lambda u_2+(\mu_{22}\phi_{u_2}-\mu_{12}\phi_{u_1})u_2\Big)u_2dx =\int_{\R^3}|\nabla u_2|^2+\lambda u_2^2 +\mu_{22}\phi_{u_2}u_2^2-\mu_{12}\phi_{u_1}u_2^2dx.
\end{equation}
  Note that, since $\mu_{ij}>0, \mu_{11}\mu_{22}\ge \mu_{12}^2$ and
$
\nabla \phi_{u_1}\cdot \nabla \phi_{u_2}\le \frac{\mu_{11}}{2\mu_{12}}|\nabla \phi_{u_1}|^2+\frac{\mu_{12}}{2\mu_{11}}|\nabla \phi_{u_2}|^2,
$
where $i,j=1,2$, we have
\begin{equation}\label{mupos}
 \mu_{11}|\nabla \phi_{u_1}|^2+\mu_{22}|\nabla \phi_{u_2}|^2 -2\mu_{12}\nabla\phi_{u_1}\cdot\nabla\phi_{u_2} \ge \frac{\mu_{11}\mu_{22}-\mu_{12}^2}{\mu_{11}}|\nabla \phi_{u_2}|^2\ge0.
\end{equation}
Then, by \eqref{mupos},  adding \eqref{ch5} and \eqref{ch6}, we get
\begin{align*}
0&= \int_{\R^3}|\nabla u_1|^2+|\nabla u_2|^2+\lambda u_1^2 +\lambda u_2^2+\mu_{11}\phi_{u_1}u_1^2+\mu_{22}\phi_{u_2}u_2^2 -2\mu_{12}\phi_{u_1}u_2^2 dx \\
& = \int_{\R^3}|\nabla u_1|^2+|\nabla u_2|^2+\lambda u_1^2 +\lambda u_2^2+\mu_{11}|\nabla \phi_{u_1}|^2+\mu_{22}|\nabla \phi_{u_2}|^2 -2\mu_{12}\nabla\phi_{u_1}\cdot\nabla\phi_{u_2} dx  \\
& \ge \int_{\R^3}|\nabla u_1|^2+|\nabla u_2|^2+\lambda u_1^2 +\lambda u_2^2dx,  
\end{align*}
which implies that $\vec{u} \in H^1(\R^3)\times H^1(\R^3)$ is trivial.

\subsection{Explicit form of $\vec{u}$ when $\det(\mu_{ij}) < 0$}
We next assume $\mu_{11}\mu_{22}< \mu_{12}^2$. Suppose that $\vec{u}$ is nontrivial, nonnegative and radial. 
We have already seen from Remark \ref{rmk-no-semi} that $\vec{u}$ is vectorial. 
Then the strong maximum principle applies to see that $\vec{u}$ is positive. 
We now state a lemma which clarifies a structure of $\vec{u}$.
\begin{lem}\label{betaunq}
Assume  $\lambda,   \mu_{12}  >0$, $\mu_{ii} \ge 0$ for $i=1,2$.
Then for any positive  radial solutions $\vec{u}=(u_1,u_2)$ of \eqref{bme}, we have $u_2\equiv a_0u_1$, where $a_0=\sqrt{\frac{\mu_{11}+\mu_{12}}{\mu_{22}+\mu_{12}}}$.
  \end{lem}
\begin{proof}
Let $\vec{u}=(u_1,u_2)$ be a positive radial solution of \eqref{bme}. 
By Newton’s Theorem and Holder's inequality,  we see that
\begin{equation}\label{phexp}
\phi_{u_i}(x)=\frac{1}{|x|}\int_0^\infty u_i^2(s)s^2\min\{|x|s^{-1},1\}ds\le C\|u_i\|_{H^1}^2\min\{1, |x|^{-1}\}
\end{equation}
for $i=1,2$. Thus, by the comparison principle, for any $\lambda^\prime<\lambda<\lambda^{\prime \prime}$, there exist constants $C_1, C_2>0$ such that
\[
C_1\exp(-\lambda^{\prime \prime}|x|)\le u_i(x)\le C_2\exp(-\lambda^{\prime }|x|) \mbox{ for } x\in \R^3
\]
and thus $ u_i^2/ u_j  \in H^1(\R^3)$, where $i,j=1,2$. Denote $v_2\equiv a_0^{-1}u_2$, where $a_0=\sqrt{\frac{\mu_{11}+\mu_{12}}{\mu_{22}+\mu_{12}}}$.
Then, since $\phi_{u_2}=a_0^2\phi_{v_2}$, we have
\begin{equation}\label{bbme}
\begin{cases} -\Delta u_1+\lambda u_1+(\mu_{11}\phi_{u_1}-a_0^2\mu_{12}\phi_{v_2})u_1 =0,
\\  -\Delta v_2+\lambda v_2+(a_0^2\mu_{22}\phi_{v_2}-\mu_{12}\phi_{u_1})v_2 =0.\end{cases}
\end{equation}
By dividing the first and second equations of \eqref{bbme} by $u_1$ and $v_2$, respectively,  and subtracting the second equation from
the first equation,  we have
\begin{align*}
0&=-(\Delta u_1)u_1^{-1}+(\Delta v_2)v_2^{-1}+\Big((\mu_{11}+\mu_{12})\phi_{u_1}-a_0^2(\mu_{22}+\mu_{12})\phi_{v_2}\Big) \\
&=-(\Delta u_1)u_1^{-1}+(\Delta v_2)v_2^{-1}+(\mu_{11}+\mu_{12}) (\phi_{u_1}-  \phi_{v_2} ).
\end{align*}
Multiplying the above equation by $u_1^2-v_2^2$, then   integration by parts yields
\begin{align*}
0&=\int_{\R^3}\nabla u_1\cdot \nabla \Big(\frac{u_1^2-v_2^2}{u_1}\Big)-\nabla v_2\cdot \Big(\frac{u_1^2-v_2^2}{v_2}\Big)+(\mu_{11}+\mu_{12})  (\phi_{u_1}-  \phi_{v_2} )(u_1^2-v_2^2) \\
&=\int_{\R^3}(u_1^2+v_2^2)\Big|\frac{\nabla u_1}{u_1}-\frac{\nabla v_2}{v_2}\Big|^2+(\mu_{11}+\mu_{12})  |\nabla( \phi_{u_1}- \phi_{v_2}) |^2dx,
\end{align*}
which implies that $u_1\equiv v_2$.
\end{proof}
By applying Lemma \ref{betaunq}, we can immediately see that $\vec{u}$ has the form  $\vec{u}=(V,a_0V)$, where $a_0=\sqrt{\frac{\mu_{11}+\mu_{12}}{\mu_{22}+\mu_{12}}}$ and $V$ is a  positive solution of the Hartree equation
\begin{equation}\label{bme11} 
-\Delta u+\lambda u+ \frac{\mu_{11}\mu_{22}-\mu_{12}^2}{\mu_{22}+\mu_{12}} \phi_{u} u=0.
\end{equation} 
We finally refer to \cite[Theorem 10]{L} or \cite[Lemma 9]{Le} to conclude that if any positive radial solution of \eqref{bme11} is unique when $\mu_{11}\mu_{22}< \mu_{12}^2$. 
This completes the proof of Theorem \ref{btm}.

\section{Structure theorems for nonnegative solutions in power potential case}\label{strucsect}
The rest of the paper is devoted to the study of the system \eqref{gme}, that is, the power potential case.
We begin with defining the energy functional for \eqref{gme}:
\begin{align*}
I(\vec{u})&=\frac12\int_{\R^3}|\nabla u_1|^2+|\nabla u_2|^2+\lambda u_1^2+\lambda u_2^2dx+\frac14\int_{\R^3} \mu_{11}u_1^2 \phi_{u_1} + \mu_{22}u_2^2 \phi_{u_2}-2\mu_{12} u_1^2 \phi_{u_2} dx\\
&\qquad - \frac{1}{p+1}\frac{1}{2\pi} \int_{\R^3}\int_0^{2\pi} |u_1+e^{i\theta}u_2|^{p+1}d\theta dx\\
&=\frac12\int_{\R^3}|\nabla u_1|^2+|\nabla u_2|^2+\lambda u_1^2+\lambda u_2^2dx+\frac14\int_{\R^3} \mu_{11}u_1^2 \phi_{u_1} + \mu_{22}u_2^2 \phi_{u_2}-2\mu_{12} u_1^2 \phi_{u_2} dx\\
&\qquad - \frac{1}{p+1} \frac{1}{2\pi}\int_{\R^3}\int_0^{2\pi} \Big(u_1^2+2u_1 u_2 \cos \theta+u_2^2\Big)^\frac{p+1}{2}d\theta dx.
\end{align*}
We prove in Section \ref{regularity} that $I$ is twice continuously differentiable on $\mathbf{H}$ and its critical points are solutions to \eqref{gme}.
By finding nonzero critical points of $I$, one therefore can provide several nontrivial solutions to \eqref{gme}. This shall be done throughout subsequent sections.

One fundamental task of this paper is to distinguish these solutions from the semi-trivial ones. 
We will see that this can be accomplished by obtaining some information on mathematical structures of the semi-trivial solutions such as the energy levels or Morse indices of them. 
The main purpose of this section is to investigate such structures of semi-trivial solutions. 
We also prove a rigidity theorem for positive solutions, which are required to give a classification on positive solutions.

\subsection{Morse indices of semi-trivial solutions}
In this subsection, we compute a lower bound of Morse indices of semi-trivial critical points of $I$.
We say $n$ is the Morse index of a critical point $\vec{u}$ of $I$ if $n$ is the maximal dimension of subspaces $V \subset \mathbf{H}$ satisfying
\[
I^{\prime \prime}(\vec{u})[h, h] < 0 \quad \text{for all } h \in V.
\]

The following identity is straightforward from \eqref{hss1} in Appendix \ref{regularity} but frequently invoked throughout this section. 
\begin{lem}\label{semimorse}
Assume $1<p<5$. For  $i=1,2$, let  $\lambda>0, \mu_{ij} \ge0, (t_1,t_2)\in \R^2$ and $(u_1,u_2)\in {\bf H}$.  Then we have
\begin{equation}\label{abhss1}
\begin{aligned}
I^{\prime \prime}(u_1,0)\Big[(t_1u_1, t_2u_1),(t_1u_1, t_2u_1)\Big]&=t_1^2\int_{\R^3} |\nabla u_1|^2+\lambda   u_1^2+3\mu_{11}    \phi_{u_1} u_1^2- p |u_1 |^{p+1}dx\\
&\quad +t_2^2\int_{\R^3} |\nabla u_1|^2+\lambda   u_1^2-\mu_{12}    \phi_{u_1} u_1^2 - \frac{p+1}{2} |u_1|^{p+1}dx
\end{aligned}
\end{equation}
and
\begin{align*}
I^{\prime \prime}(0,u_2)\Big[(t_1u_2, t_2u_2),(t_1u_2, t_2u_2)\Big] &=t_1^2\int_{\R^3} |\nabla u_2|^2+\lambda   u_2^2-\mu_{12}    \phi_{u_2} u_2^2 - \frac{p+1}{2} |u_2|^{p+1}dx\\
&\quad +t_2^2\int_{\R^3} |\nabla u_2|^2+\lambda   u_2^2+3\mu_{22}    \phi_{u_2} u_2^2- p |u_2 |^{p+1}dx.
\end{align*} 
\end{lem}

\begin{prop}\label{mmors}
Assume that  $1<p<5$. Let  $\lambda$ and $\mu_{ij}$ be positive constants for $i=1,2$.
Then every semi-trivial critical point of $I$ has Morse index larger than or equal to $1$.  
\end{prop}
\begin{proof}
We may only deal with the semi-trivial critical point $(u_1,0)$ due to the symmetry of the system \eqref{gme}.
Since $u_2 = 0$, $u_1$ solves
$$-\Delta u_1+\lambda u_1+ \mu_{11}\phi_{u_1} u_1 =  |u_1|^{p-1} u_1$$
so that one has
$$
\int_{\R^3} |\nabla u_1|^2+\lambda   u_1^2+\mu_{11}    \phi_{u_1} u_1^2 -   |u_1|^{p+1}dx=0. $$
Then by \eqref{abhss1}, we see that for $t\in \R\setminus \{0\}$,
\begin{equation*}
\begin{aligned}
&I^{\prime \prime}(u_1,0)\Big[(0, tu_1),(0, tu_1)\Big] \\
&= t^2\int_{\R^3} |\nabla u_1|^2+\lambda   u_1^2-\mu_{12}    \phi_{u_1} u_1^2 - \frac{p+1}{2} |u_1|^{p+1}dx\\
&=-t^2\int_{\R^3}   (\mu_{11}+\mu_{12}) \phi_{u_1} u_1^2+\frac{p-1}{2} |u_1|^{p+1}dx<0.
\end{aligned}
\end{equation*}
This shows that $1$ is a lower bound of  Morse index of $(u_1,0)$. 
\end{proof}

\begin{prop}\label{msi1}
Let  $\lambda$ and $\mu_{ij}$ be positive constants for $i=1,2$. Assume that  $\frac13 (-2 + \sqrt{73})\le p<5$.  
Then every semi-trivial critical point of $I$ has Morse index larger than or equal to $2$.  
\end{prop}
\begin{proof}
Let $(u_1,0)$ be a semi-trivial critical point of $I$ so that $u_1$ solves 
$$-\Delta u_1+\lambda u_1+ \mu_{11}\phi_{u_1} u_1 =  |u_1|^{p-1} u_1.$$ 
Define
$
a= \int_{\R^3} |\nabla u_1|^2dx,\ \  b=  \int_{\R^3}\lambda   u_1^2,\ \  c= \int_{\R^3}\mu_{11}    \phi_{u_1} u_1^2,\ \  d= \int_{\R^3} |u_1 |^{p+1}dx.
$
Then we have
\begin{equation}\label{esee1}
\begin{cases}
a+b+c-d=0,\\
\frac12 a+\frac32 b+\frac54 c-\frac{3}{p+1}d=0.
\end{cases}
\end{equation}
Here, the second one is so-called the Pohozaev identity. We refer to \cite{R} or Proposition \ref{pz}  applied to the solution $(u_1,0)$ in Appendix.

By \eqref{abhss1}, we have
\begin{equation}\label{ese2}
\begin{aligned}
I^{\prime \prime}(u_1,0)\Big[(t_1u_1, t_2u_1),(t_1u_1, t_2u_1)\Big] 
&=t_1^2\int_{\R^3} |\nabla u_1|^2+\lambda   u_1^2+3\mu_{11}    \phi_{u_1} u_1^2- p |u_1 |^{p+1}dx\\
&\quad +t_2^2\int_{\R^3} |\nabla u_1|^2+\lambda   u_1^2-\mu_{12}    \phi_{u_1} u_1^2 - \frac{p+1}{2} |u_1|^{p+1}dx\\
&=t_1^2 (\rom{1})+t_2^2(\rom{2}).
\end{aligned}
\end{equation}
Since $a+b+c-d=0$, we have
\begin{equation}\label{ese3}
\begin{aligned}
(\rom{2})=-\int_{\R^3}   (\mu_{11}+\mu_{12}) \phi_{u_1} u_1^2+\frac{p-1}{2} |u_1|^{p+1}dx<0.
\end{aligned}
\end{equation}

On the other hand, by \eqref{esee1},
we have
$$
a=\frac{1}{3}b+\frac{5p-7}{3(p+1)}d,\ \  c=-\frac{4}{3}b+\frac{2(5-p)}{3(p+1)}d
$$
and
$$
(\rom{1})=a + b + 3 c - p d=\frac13 \Big(-8b-\frac{3p^2+4p-23}{p+1}d\Big).
$$
Thus,  $\frac{3p^2+4p-23}{p+1}>0$ if $p>\frac13 (-2 + \sqrt{73})\cong 2.18133$, which implies that $(\rom{1})<0 $ if $p\ge\frac13 (-2 + \sqrt{73})$
and consequently the Morse index of $(u_1,0)$ is larger than or equal to 2.
By the symmetry, Morse index of the semi-trivial solution $(0,u_2)$ for \eqref{gme} is also larger than or equal to two.
\end{proof}

\subsection{Energy comparison of semi-trivial solutions}
In this subsection, we prove that a semi-trivial solution to \eqref{gme} is not a ground state when $\mu_{11} = \mu_{22}$ and $2 < p < 5$.
We accomplish this by comparing the energy of semi-trivial solutions with the energy of the solution of the form $(W,W)$, whose existence is guaranteed for the case $\mu_{11} = \mu_{22}$.
In fact, if $\mu_{11} = \mu_{22}$ then $(W,W)$ gives a solution of \eqref{gme} for any solution $W$ of the equation
\[
 -\Delta u+ \lambda  u + ( {\mu_{11} - \mu_{12}}) \phi_{u} u= \frac{c_p}{2}|u|^{p-1} u,
\]
where we define
\begin{equation}\label{cp}
\begin{aligned}
c_p \coloneqq \frac{1}{2\pi}\int_0^{2\pi} |1+e^{i\theta}|^{p+1}d\theta=\frac{2^{\frac{p-1}{2}}}{ \pi}\int_0^{2\pi} (1+\cos \theta)^{\frac{p+1}{2}}d\theta. 
\end{aligned}
\end{equation}
  We shall show that the energy level of  $(W,W)$ is lower than that of semi-trivial solutions. 
  
We first prepare some auxiliary lemmas. 
\begin{lem}\label{uniqct}\cite[Lemma 3.3]{R}
Let $2<p<5$, $a,c>0$ and $b\in \R$. Define $h:[0,\infty)\rightarrow \R$ by
$$
h(t)=at+bt^3-ct^{2p-1}.
$$
Then $h$ has a unique critical point, which corresponds to its maximum.
\end{lem}

Let us denote
$$
I_\gamma(u) \coloneqq \int_{\R^3}\frac12\Big(|\nabla u|^2+u^2\Big)+\frac{\gamma}{4}  u^2 \phi_{u} -\frac{1}{p+1} |u|^{p+1}dx.
$$
\begin{lem}\label{l1}
Let $v$ and $w$ be  positive ground state solutions of 
\[
 -\Delta v+v+ \gamma \phi_{v} v= |v|^{p-1} v \ \mbox{ and }\ -\Delta w+w+ \mu \phi_{w} w = |w|^{p-1} w,
\]
respectively, where $2<p<5$ and $  \gamma, \mu\in \R$ with $\gamma>\mu. $
Then we have
$$
I_\gamma(v) > I_\mu(w).
$$
\end{lem}
\begin{rmk}\em
The existence of a positive radial ground states for $I_\gamma$ with $2< p < 5$ was proved in \cite{R} when $\gamma > 0$ and in \cite{V} when $\gamma < 0$.
\end{rmk}
\begin{proof}
In fact, we recall 
$$
I_\gamma(v)=\inf_{u\not\neq0}\max_{t\ge 0}I_\gamma(u_t) \mbox{ and } I_\mu(w)=\inf_{u\not\neq0}\max_{t\ge 0} I_\mu(u_t),
$$
where $u_t(x)\equiv t^2u(tx)$ (see \cite[Lemma 2.4]{AP}). 
Combining the Nehari's identity  
$$
\int_{\R^3} |\nabla v|^2+  v^2+\gamma v^2\phi_v- |v|^{p+1}dx=0
$$
and the Pohozaev's identity (see \cite{R} or Proposition \ref{pz})
$$
\int_{\R^3}\frac12 |\nabla v|^2+\frac32 v^2+\frac{5\gamma}{4}v^2\phi_v-\frac{3}{p+1}v^{p+1}dx=0,
$$
we get
$$
\int_{\R^3}\frac32 |\nabla v|^2+\frac12 v^2+\frac{3\gamma}{4}v^2\phi_v-\frac{2p-1}{p+1}v^{p+1}dx=0.
$$ 
From this and Lemma \ref{uniqct}, we see that $I_\gamma(v) =\max_{t\ge 0}I_\gamma(v_t)$, and thus
\begin{align*}
I_\gamma(v)&=\max_{t\ge 0}I_\gamma(v_t) =\max_{t\ge 0}\bigg[\int_{\R^3}\frac12\Big(t^3|\nabla v|^2+tv^2\Big)+\frac{\gamma}{4}t^3  v^2 \phi_{v} -\frac{t^{2p-1}}{p+1} |v|^{p+1}dx\bigg]\\
&>\max_{t\ge 0}\bigg[\int_{\R^3}\frac12\Big(t^3|\nabla v|^2+tv^2\Big)+\frac{\mu}{4}t^3  v^2 \phi_{v} -\frac{t^{2p-1}}{p+1} |v|^{p+1}dx\bigg] \ge I_\mu(w).
\end{align*}
\end{proof}

A direct computation shows that there holds the following scaling properties.   
\begin{lem}\label{rmk1}
Let $v$ be a solution of 
$$
-\Delta v+a v+ b\phi_{v} v= c|v|^{p-1} v,
$$ where $1<p<5$ and $a,b,c>0$, 
which is a Euler-Lagrange equation of 
\[
I_{a,b,c}(v) = \int_{\R^3}\frac12|\nabla v|^2+\frac{a}{2}v^2+\frac{b}{4}\phi_v v^2-\frac{c}{p+1}|v|^{p+1}dx.
\]
Then a rescaled function $\tilde{v}(x)=\Big(\frac{c}{a}\Big)^{\frac{1}{p-1}}v\Big(\frac{x}{\sqrt{a}}\Big)$ solves
$$
-\Delta \tilde{v}+  \tilde{v}+ \gamma(a,b,c)\phi_{\tilde{v}} \tilde{v}=  |\tilde{v}|^{p-1} \tilde{v},
$$
where  $\gamma(a,b,c) = \frac{b}{a^2}\left(\frac{a}{c}\right)^\frac{2}{p-1}$
and satisfies
\[
I_{\gamma(a,b,c)}(\tilde{v}) = I_{1,\gamma(a,b,c),1}(\tilde{v}) = \Big(\frac{c}{a}\Big)^{\frac{2}{p-1}}a^\frac{1}{2} I_{a,b,c}(v).
\]
\end{lem}


Now, we are ready to prove the main result of this subsection. 
\begin{prop}\label{egc1}
Assume $2<p<5$, $\lambda>0, \mu_{ij}>0$,  $\mu_{11}=\mu_{22}$, where $i,j =1,2$. Then we have
\begin{equation}\label{engc}
  I(V,0)=I(0,V) > I(W,W),
\end{equation}
where $W$ and  $V$ are  positive ground state solutions of 
$$
 -\Delta u+ \lambda  u + ( {\mu_{11} - \mu_{12}}) \phi_{u} u= \frac{c_p}{2}|u|^{p-1} u \ \mbox{ and }\  -\Delta u+\lambda u+ \mu_{11}\phi_{u} u=  |u|^{p-1} u,  
$$  
respectively. 
\end{prop}

\begin{proof}
 Denote $\tilde{W}(x)=\Big(\frac{c_p}{2\lambda }\Big)^\frac{1}{p-1}W\Big(\frac{x}{\sqrt{\lambda}}\Big) $ and $\tilde{V}(x)=\Big(\frac{1}{\lambda }\Big)^\frac{1}{p-1}V \Big(\frac{x}{\sqrt{\lambda}}\Big)$. Then, by Lemma \ref{rmk1}, we see that $\tilde{W}$ and $\tilde{V} $ are  positive ground state solutions of 
$$
-\Delta u+u +\frac{ 2^\frac{2}{p-1}(\mu_{11}-\mu_{12})}{ \lambda^\frac{2p-4}{p-1} c_p^\frac{2}{p-1}}\phi_{u} u= |u|^{p-1} u \ \mbox{ and }\ -\Delta u+u+ \frac{\mu_{11}}{\lambda^\frac{2p-4}{p-1}  } \phi_{u} u= |u|^{p-1} u, 
$$
respectively. By Jensen's inequality, we see that  for $1<p<5$, 
$$
1=\left(\frac{1}{2\pi}\int_0^{2\pi}(1+\cos\theta)d\theta\right)^\frac{p+1}{2}<\frac{1}{2\pi}\int_0^{2\pi}(1+\cos\theta)^\frac{p+1}{2}d\theta.
$$
which implies that $2^{\frac{p+1}{2}}< c_p$ for $1<p<5$.
Note that this means $\Big(\frac{c_p}{2}\Big)^\frac{2}{p-1}> 1$ for $2<p<5$ and thus  we have $ \mu_{11} >\frac{ 2^\frac{2 }{p-1}(\mu_{11} - \mu_{12}) }{  c_p^\frac{2}{p-1}}$.
Then, by   Lemma \ref{l1} and Lemma \ref{rmk1}, 
\begin{align*}
&\int_{\R^3}\frac12\Big(|\nabla \tilde{V}|^2+\tilde{V}^2\Big)+\frac{\mu_{11}}{4\lambda ^\frac{2p-4}{p-1} }  \tilde{V}^2 \phi_{\tilde{V}} -\frac{1}{p+1} |\tilde{V}|^{p+1}dx\\
&= \lambda^{-\frac{5-p}{2(p-1)}}\int_{\R^3}\frac12\Big(|\nabla V|^2+\lambda V^2\Big)+\frac{\mu_{11}}{4}  V^2 \phi_{V} -\frac{1}{p+1} |V|^{p+1}dx\\
&=\lambda^{-\frac{5-p}{2(p-1)}}I(V,0) \\
&> \int_{\R^3}\frac12\Big(|\nabla \tilde{W}|^2+\tilde{W}^2\Big) +\frac{ 2^{\frac{2}{p-1}}(\mu_{11} - \mu_{12}) }{4\lambda^\frac{2p-4}{p-1} c_p^\frac{2}{p-1}}\phi_{\tilde{W}}\tilde{W}^2  -\frac{1}{p+1} |\tilde{W}|^{p+1}dx \\
&=\frac{c_p^\frac{2}{p-1} }{ 2^\frac{2}{ p-1 }\lambda^\frac{5-p}{2(p-1)}}   \int_{\R^3}\frac12\Big(|\nabla W|^2+ \lambda W^2\Big) +\frac{\mu_{11}-\mu_{12}}{4}\phi_W W^2  -\frac{c_p}{2(p+1)}  |W|^{p+1}dx\\
&=\frac{c_p^\frac{2}{p-1} }{ 2^\frac{2}{ p-1 }  \lambda^\frac{5-p}{2(p-1)}}\frac12 I(W,W).
\end{align*}
Note that, since
$2^{\frac{p+1}{2}}< c_p$ for $2<p<5$, we get $  \lambda^{-\frac{5-p}{2(p-1)}} <\frac{c_p^\frac{2}{p-1} }{2^{ \frac{p+1}{p-1}} \lambda^{ \frac{5-p}{2(p-1)}}}.$
From this,
we see that for $2<p<5$,
$$
  I(V,0)= I(0,V)  > I(W,W).
$$
\end{proof}

\subsection{Rigidity results for positive solutions}\label{rigid}
Analogously with Lemma \ref{betaunq}, we can show the coincidence of components for positive solutions to \eqref{gme} with $\mu_{11} = \mu_{22}$.
\begin{prop}\label{rig1}
Assume $1< p<5$, $\lambda>0$, $\mu_{ij}\ge0$ and  $\mu_{11}=\mu_{22}$, where $i,j=1,2$. For any positive radial solutions $\vec{u}=(u_1,u_2)$ of \eqref{gme}, we have $u_1\equiv u_2$.
\end{prop}
\begin{proof}
Let $\vec{u}=(u_1,u_2)\in {\bf H}$ be a positive solution  of 
$$
\begin{cases} -\Delta u_1+\lambda u_1+(\mu_{11}\phi_{u_1}-\mu_{12}\phi_{u_2})u_1=\frac{1}{2\pi}\int_0^{2\pi}(u_1+u_2 \cos \theta )\Big(u_1^2+2u_1u_2 \cos \theta+u_2^2\Big)^\frac{p-1}{2}d\theta,
\\  -\Delta u_2+\lambda u_2+(\mu_{22}\phi_{u_2}-\mu_{12}\phi_{u_1})u_2=\frac{1}{2\pi}\int_0^{2\pi}(u_2+u_1 \cos \theta )\Big(u_1^2+2u_1u_2 \cos \theta+u_2^2\Big)^\frac{p-1}{2}d\theta. \end{cases}
$$
 By \eqref{ibp}, we see that
\begin{align*}
\int_0^{2\pi}(u_1+u_2 \cos \theta )\Big(u_1^2+2u_1u_2 \cos \theta+u_2^2\Big)^\frac{p-1}{2}d\theta\le Cu_1(u_1^{p-1}+u_2^{p-1})\le \tilde{C}|x|^{-(p-1)}u_1,
\end{align*}
where we used  the Strauss inequality $|u_i(x)|\le C\|u_i\|_{H^1}|x|^{-1}$ for $i=1,2$. Moreover, by    Newton’s Theorem, 
$$
\phi_{u_i}(x)=\frac{1}{|x|}\int_0^\infty u_i^2(s)s\min\{|x|,s\}ds\le C\min\{1, \|u_i\|_{H^1}^2|x|^{-1}\}
$$
 for $i=1,2$. Thus, by the comparison principle, for any $\lambda^\prime<\lambda<\lambda^{\prime \prime}$, there exist constants $C_1, C_2>0$ such that
\begin{equation}\label{uqz2}
C_1\exp(-\lambda^{\prime \prime}|x|)\le u_i(x)\le C_2\exp(-\lambda^{\prime }|x|) \mbox{ for } x\in \R^3,
\end{equation}
where $i=1,2$. From \eqref{uqz2}, we deduce that $ u_i^2/ u_j  \in H^1(\R^3)$, where $i,j=1,2$. Moreover,  by \eqref{uq}, we see that for $a,b>0$,
\begin{equation}\label{uqz1}
\begin{aligned}
&\int_0^{2\pi}\frac{a+b \cos \theta }{a}\Big(a^2+2ab \cos \theta+b^2\Big)^\frac{p-1}{2} -\frac{b+a \cos \theta }{b}\Big(a^2+2ab \cos \theta+b^2\Big)^\frac{p-1}{2}d\theta\\
&=  \frac{b^2-a^2}{ab}\int_0^{2\pi} \cos\theta \Big(a^2+2ab \cos \theta+b^2\Big)^\frac{p-1}{2} d\theta\\
&=  (p-1)(b^2 -a^2 )\int_0^{2\pi} \sin^2\theta \Big(a^2+2ab \cos \theta+b^2\Big)^\frac{p-3}{2} d\theta.
\end{aligned}
\end{equation}
Then, by dividing the first and second equations of \eqref{gme} by $u_1$ and $u_2$, respectively,  and subtracting the second one from the first one,   we see from  \eqref{uqz1} that
\begin{align*}
0&=-(\Delta u_1)u_1^{-1}+(\Delta u_2)u_2^{-1}+\Big((\mu_{11}+\mu_{12})\phi_{u_1}-(\mu_{22}+\mu_{12})\phi_{u_2}\Big)\\
&\qquad +\frac{1}{2\pi }\int_0^{2\pi}  \bigg[\frac{u_2+u_1 \cos \theta }{u_2}-\frac{(u_1+u_2 \cos \theta )}{u_1}\bigg](u_1^2+2u_1u_2 \cos \theta+u_2^2 )^\frac{p-1}{2} d\theta \\
&=-(\Delta u_1)u_1^{-1}+(\Delta u_2)u_2^{-1}+(\mu_{11}+\mu_{12}) (\phi_{u_1}- \phi_{u_2} )\\
&\quad +\frac{(p-1)( u_1^2-u_2^2 )}{2\pi  }\int_0^{2\pi} \sin^2\theta \Big(u_1^2+2u_1u_2 \cos \theta+u_2^2\Big)^\frac{p-3}{2} d\theta.
\end{align*}
By \eqref{uqz2} and the fact that
$$
(u_1^2-u_2^2)^2(u_1^2+2u_1u_2 \cos \theta+u_2^2 )^\frac{p-3}{2}\le C\begin{cases}(u_1+u_2)^2|u_1 -u_2 |^{p-1} &\mbox{ if } 1<p\le 3,\\
(u_1^4+u_2^4)(u_1^{p-3}+u_2^{p-3}) &\mbox{ if } 3<p<5,
\end{cases}
$$  multiplying the above equation by $u_1^2-u_2^2$, then  integration by parts yields 
\begin{equation}\label{eqid}
\begin{aligned}
0&=\int_{\R^3}\nabla u_1\cdot \nabla \Big(\frac{u_1^2-u_2^2}{u_1}\Big)-\nabla u_2\cdot \Big(\frac{u_1^2-u_2^2}{u_2}\Big)+(\mu_{11}+\mu_{12})  (\phi_{u_1}- \phi_{u_2} )(u_1^2-u_2^2)\\
&\qquad + \frac{(p-1)(u_1^2-u_2^2)^2}{2\pi  }\int_0^{2\pi} \sin^2\theta \Big(u_1^2+2u_1u_2 \cos \theta+u_2^2\Big)^\frac{p-3}{2} d\theta dx\\
&=\int_{\R^3}(u_1^2+u_2^2)\Big|\frac{\nabla u_1}{u_1}-\frac{\nabla u_2}{u_2}\Big|^2+(\mu_{11}+\mu_{12})  |\nabla( \phi_{u_1}- \phi_{u_2}) |^2\\
&\qquad + \frac{(p-1)(u_1^2-u_2^2)^2}{2\pi  }\int_0^{2\pi} \sin^2\theta \Big(u_1^2+2u_1u_2 \cos \theta+u_2^2\Big)^\frac{p-3}{2} d\theta dx.
\end{aligned}
\end{equation}
If $u_1\not\equiv u_2$ in $\R^3$, then we may assume that $\Omega=\{x\in \R^3 : u_1(x)<u_2(x)\}\neq \emptyset.$ Then by \eqref{eqid} and the fact that 
$$
u_1^2+2u_1u_2 \cos \theta+u_2^2=(u_2-u_1)^2+2u_1u_2(1+\cos\theta)>0 \mbox{ for all }\theta\in [0,2\pi] \mbox{ and }x\in \Omega,$$ we see that $u_1\equiv u_2$ in $\Omega$, which is a contradiction. Thus, we have $u_1 \equiv u_2$ in $\R^3$.
\end{proof}

\section{Power potential case with $2<p<5$}\label{nonlin} 
 In this section, we prove the existence results of \eqref{gme} for $2 < p < 5$, and give the proof of Theoreom  \ref{th3}.
\subsection{Construction of a nonnegative ground state}\label{p25}
In this subsection, we construct a nontrivial nonnegative ground state solution of \eqref{gme} for $2<p<5$. 
We mainly follow the approach adopted in \cite{R} but we need to modify arguments because the term 
\[
\int_{\R^3}\mu_{11}u_1^2 \phi_{u_1}+\mu_{22}u_2^2 \phi_{u_2} -2 \mu_{12} u_1^2 \phi_{u_2}\,dx
\]  
appearing in the energy functional $I$ is not sign definite. 
 
We begin with by setting the  manifold
\begin{equation}\label{mmf}
\mathcal{M}\equiv \{ \vec{u}=(u_1,u_2)\in {\bf H} \setminus \{\vec{0}\} \ |\ G(\vec{u})=0 \},
\end{equation}
where 
\begin{align*}
G(\vec{u})&\equiv 2J(\vec{u})+P(\vec{u})\\
&= \int_{\R^3} \frac32\Big(|\nabla u_1|^2+|\nabla u_2|^2\Big)+\frac{1}{2}\lambda( u_1^2+ u_2^2)+\frac{3}{4} \Big(\mu_{11}u_1^2 \phi_{u_1}+\mu_{22}u_2^2 \phi_{u_2} -2 \mu_{12} u_1^2 \phi_{u_2} \Big)  \\
&\qquad -\frac{2p-1}{p+1}\bigg[ \frac{1}{2\pi}\int_0^{2\pi} \Big(u_1^2+2u_1 u_2 \cos \theta+u_2^2\Big)^\frac{p+1}{2}d\theta\bigg] dx
\end{align*}
and $J$ and $P$ are given by
\begin{align*}
J(\vec{u})&\equiv I^\prime(\vec{u})\vec{u}\\ 
&=\int_{\R^3}|\nabla u_1|^2+|\nabla u_2|^2+\lambda u_1^2+\lambda u_2^2dx+\int_{\R^3} \mu_{11}u_1^2 \phi_{u_1} + \mu_{22}u_2^2 \phi_{u_2}-2\mu_{12} u_1^2 \phi_{u_2} dx\\
&\qquad -  \frac{1}{2\pi}  \int_{\R^3}\int_0^{2\pi}\Big(u_1^2+2u_1 u_2 \cos \theta+u_2^2\Big)^\frac{p+1}{2} d\theta dx,
\end{align*}
\begin{align*}
P(\vec{u})&\equiv \int_{\R^3}-\frac12\Big(|\nabla u_1|^2+|\nabla u_2|^2\Big)-\frac{3}{2}\lambda( u_1^2+ u_2^2)dx-\frac{5}{4} \int_{\R^3}\mu_{11}u_1^2 \phi_{u_1}+\mu_{22}u_2^2 \phi_{u_2}  -2 \mu_{12} u_1^2 \phi_{u_2} dx  \\
&\qquad + \frac{3}{p+1}\int_{\R^3}\frac{1}{2\pi}\int_0^{2\pi} \Big(u_1^2+2u_1 u_2 \cos \theta+u_2^2\Big)^\frac{p+1}{2}d\theta dx. 
\end{align*}

Note that, since
\begin{equation}\label{gx}
\begin{aligned}
&G^\prime(\vec{u})[\psi_1,\psi_2]\\
&= \int_{\R^3} 3(\nabla u_1\cdot \nabla \psi_1 + \nabla u_2\cdot \nabla \psi_2)+\lambda u_1 \psi_1+\lambda u_2\psi_2dx\\
&\quad + 3\int_{\R^3}\mu_{11}   u_1 \phi_{u_1} \psi_1 + \mu_{22}    u_2  \phi_{u_2} \psi_2 - \mu_{12} u_1  \phi_{u_2} \psi_1 -\mu_{12}  u_2 \phi_{u_1} \psi_2 dx\\
&\quad - (2p-1) \frac{1}{2\pi}\int_{\R^3}\int_0^{2\pi}\Big((u_1+u_2\cos \theta ) \psi_1+( u_1\cos \theta+u_2) \psi_2\Big)\Big(u_1^2+2u_1 u_2 \cos \theta+u_2^2\Big)^\frac{p-1}{2} d\theta dx\\
&= \int_{\R^3} 3(\nabla u_1\cdot \nabla \psi_1 + \nabla u_2\cdot \nabla \psi_2)+\lambda u_1 \psi_1+\lambda u_2\psi_2dx\\
&\quad + 3\int_{\R^3}\mu_{11}   u_1 \phi_{u_1} \psi_1 + \mu_{22}    u_2  \phi_{u_2} \psi_2 - \mu_{12} u_1  \phi_{u_2} \psi_1 -\mu_{12}  u_2 \phi_{u_1} \psi_2 dx\\
&\quad - (2p-1) \frac{1}{2\pi}\int_{\R^3}\int_0^{2\pi}|u_1+e^{i\theta}u_2|^{p-1}(u_1+e^{i\theta}u_2) \psi_1+|u_2+e^{i\theta}u_1|^{p-1}(u_2+e^{i\theta}u_1) \psi_2 d\theta dx,
\end{aligned}
\end{equation}
  if $G^\prime(\vec{u})=0$, $\vec{u}$ satisfies
\begin{equation}\label{eqg}
\begin{cases} -3\Delta u_1+\lambda u_1+3(\mu_{11}\phi_{u_1}-\mu_{12}\phi_{u_2})u_1=(2p-1)\frac{1}{2\pi}\int_0^{2\pi} |u_1+e^{i\theta}u_2|^{p-1}(u_1+e^{i\theta}u_2)  d\theta,
\\  -3\Delta u_2+\lambda u_2+3(\mu_{22}\phi_{u_2}-\mu_{12}\phi_{u_1})u_2=(2p-1)\frac{1}{2\pi}\int_0^{2\pi}|u_2+e^{i\theta}u_1|^{p-1}(u_2+e^{i\theta}u_1) d\theta. \end{cases}
\end{equation}

\begin{lem}\label{constr1}
Let $2<p<5$ and $\lambda, \mu_{ij} >0$ for  $i,j=1,2$. 
Then $\mathcal{M}$ is non-empty  and $\vec{0}\notin \partial\mathcal{M}$. Moreover, $\inf_{\mathcal{M}} I >0$, and if $\inf_{\mathcal{M}} I=I(\vec{u})$,  $G^\prime(\vec{u}) \neq 0$.
\end{lem}
\begin{proof}

Note that, by Jensen's inequality,  for $\vec{u}=(u_1,u_2)\in {\bf H}\setminus \{\vec{0}\}$
\begin{equation}\label{je1}
\begin{aligned}
0<\int_{\R^3}(u_1^2+u_2^2)^\frac{p+1}{2}dx&=\int_{\R^3}\Big(\frac{1}{2\pi}\int_0^{2\pi}  u_1^2+2 u_1 u_2 \cos \theta+ u_2^2d\theta\Big)^\frac{p+1}{2}dx\\
&\le\frac{1}{2\pi}\int_{\R^3}\int_0^{2\pi} \Big( u_1^2+2 u_1 u_2 \cos \theta+ u_2^2\Big)^\frac{p+1}{2}d\theta dx.
\end{aligned}
\end{equation} 
For any $\vec{u}=(u_1,u_2)\in {\bf H}\setminus \{\vec{0}\}$, one has that
\begin{align*}
&t^2(u_1(t\cdot),u_2(t\cdot))\in \mathcal{M} \\
& \iff \int_{\R^3} \frac32 t^3\Big(|\nabla u_1|^2+|\nabla u_2|^2\Big)+\frac{1}{2}t\lambda( u_1^2+ u_2^2)+\frac{3}{4} t^3\Big(\mu_{11}u_1^2 \phi_{u_1}+\mu_{22}u_2^2 \phi_{u_2} -2 \mu_{12} u_1^2 \phi_{u_2} \Big)dx \\
&\qquad \qquad =t^{2p-1}\frac{2p-1}{p+1}\int_{\R^3}\bigg[ \frac{1}{2\pi}\int_0^{2\pi} \Big(u_1^2+2u_1 u_2 \cos \theta+u_2^2\Big)^\frac{p+1}{2}d\theta\bigg] dx.
\end{align*}
From this and \eqref{je1}, for all $\vec{u}=(u_1,u_2)\in {\bf H}\setminus \{\vec{0}\}$, there exists a unique $t>0$ such that $t\vec{u} \in \mathcal{M}$, which implies that $\mathcal{M}$ is non-empty. 

By  the facts that
$
\int_{\R^3}|u|^{p+1}  dx\le    C_1 \|u\|_{H^1}^{p+1}   
$
and 
\begin{equation}\label{ch1}
\int_{\R^3}u^2\phi_vdx\le \| u\|_{L^\frac{12}{5}}^2\|\phi_v\|_{L^6} \le   \frac12\Big( \| u\|_{L^\frac{12}{5}}^4 +\|\phi_v\|_{L^6}^2\Big)\le  C_2\Big(\|u\|_{H^1}^4 +\|v\|_{H^1}^4\Big) ,
\end{equation}
where $C_1, C_2>0$ are constants, we have
$$
G(\vec{u})\ge\frac12 (\|u_1\|_{\lambda}^2+\|u_2\|_{\lambda}^2)\left(1-C\frac{\|u_1\|_{\lambda}^4+\|u_2\|_{\lambda}^4}{\|u_1\|_{\lambda}^2+\|u_2\|_{\lambda}^2}-C\frac{\|u_1\|_{\lambda}^{p+1}+\|u_2\|_{\lambda}^{p+1}}{\|u_1\|_{\lambda}^2+\|u_2\|_{\lambda}^2}\right),
$$
and hence we deduce that there is $\rho>0$ such that
\begin{equation}\label{b1}
\|u_1\|_{\lambda}^2+\|u_2\|_{\lambda}^2\ge \rho \mbox{ for all } \vec{u}\in \mathcal{M},
\end{equation}
which implies that $\vec{0}\notin \partial\mathcal{M}$.

We claim that $\inf_{\mathcal{M}} I >0$. Let $k=I(\vec{u})$, where $\vec{u}=(u_1,u_2)\in \mathcal{M}$.
Define
\begin{equation}\label{nota1}
\begin{aligned}
a &=\int_{\R^3}  |\nabla u_1|^2+|\nabla u_2|^2dx, \ \ b=\int_{\R^3}\lambda u_1^2+\lambda u_2^2dx,  \\
c &=\int_{\R^3}\mu_{11}u_1^2 \phi_{u_1}+\mu_{22}u_2^2 \phi_{u_2} -2 \mu_{12} u_1^2 \phi_{u_2} dx, \ \ d=\int_{\R^3}\bigg[ \frac{1}{2\pi}\int_0^{2\pi} \Big(u_1^2+2u_1 u_2 \cos \theta+u_2^2\Big)^\frac{p+1}{2}d\theta\bigg] dx.
\end{aligned}
\end{equation}
Note that $a,b,c$ and $d$ satisfy
\begin{equation}\label{abcd-rel}
\begin{cases}
\frac12 a+\frac12 b+\frac14 c-\frac{1}{p+1}  d=k,\\
\frac32 a+\frac12 b+\frac34c-\frac{2p-1}{p+1}  d=0.
\end{cases}
\end{equation}
We solve the above system of equations for arbitrarily given $a,b$ and $k$ to get 
$$
c=-2\frac{a(p-2)+b(p-1)+k(1-2p)}{p-2}, \ \ d=-\frac{b(p+1)-3k(p+1)}{2(p-2)}.
$$
From the latter equality, we have
\begin{equation}\label{b3}
(p+1)b+2(p-2)d=3k(p+1)
\end{equation}
so that $k > 0$ and consequently $\inf_{\mathcal{M}}I \geq 0$.
Moreover, since $2<\frac{12}{5}<p+1$,  
\begin{align*}
\int_{\R^3}|\nabla \phi_{u}|^2dx&=\int_{\R^3}  \phi_{u} u^2dx\le \|\phi_{u}\|_{L^6}\|u\|_{L^\frac{12}{5}}^2 \le C\|\phi_{u}\|_{D^{1,2}}\|u\|_{L^2}^{2\Lambda}\|u\|_{L^{p+1}}^{2(1-\Lambda)},
\end{align*}
where $\Lambda=\frac{5p-7}{6(p-1)}$, and hence by \eqref{je1} we deduce that 
\begin{equation}\label{err1}
|c|\le C(\|\phi_{u_1}\|_{D^{1,2}}^2+\|\phi_{u_2}\|_{D^{1,2}}^2)\le C(b^{4\Lambda}+d^{\frac{8(1-\Lambda)}{p+1}}).
\end{equation}
Now, suppose that $\inf_{\mathcal{M}}I = 0$. Then there exists a sequence $\{\vec{u}_n\} \in \mathcal{M}$ satisfying
$k_n = I(\vec{u_n}) \to 0$ as $n \to \infty$. We define the constants $a_n, b_n, c_n, d_n$ as in \eqref{nota1} for $\vec{u}_n$. Since $p > 2,\, b_n,d_n >0$, \eqref{b3} tells us that $b_n,d_n \to 0$ as $n\to\infty$.
This shows by \eqref{err1} $c_n \to 0$ as $n\to \infty$ but this contradicts with \eqref{b1} and \eqref{abcd-rel}.
Thus we conclude that $\inf_{\mathcal{M}} I >0$.

Finally, we claim that  if $\inf_{\mathcal{M}} I=I(\vec{u})=k$, then $G^\prime(\vec{u}) \neq 0$. Contrary to the claim, we assume that $\inf_{\mathcal{M}} I=I(\vec{u})=k$ and   $G^\prime(\vec{u}) =0$.
By \eqref{gx}, one has
\begin{align*}
G^\prime(\vec{u})\vec{u}&= \int_{\R^3} 3 (|\nabla u_1|^2+|\nabla u_2|^2 )+ \lambda( u_1^2+ u_2^2)+3 \Big(\mu_{11}u_1^2 \phi_{u_1}+\mu_{22}u_2^2 \phi_{u_2} -2 \mu_{12} u_1^2 \phi_{u_2} \Big)  \\
&\qquad -(2p-1) \bigg[ \frac{1}{2\pi}\int_0^{2\pi} \Big(u_1^2+2u_1 u_2 \cos \theta+u_2^2\Big)^\frac{p+1}{2}d\theta\bigg]  dx=0.
\end{align*}
Then if $G^\prime(\vec{u})=0$, we have
$$
\begin{cases}
\frac12 a+\frac12 b+\frac14 c-\frac{1}{p+1}  d=k,\\
\frac32 a+\frac12 b+\frac34c-\frac{2p-1}{p+1}  d=0\\
3a+b+3c-(2p-1)d=0\\
\frac32 a+ \frac32 b+3\frac54c-(2p-1)\frac{3}{p+1}d=0.
\end{cases}
$$
The fourth equation is the Pohozaev identity(see Proposition \ref{pz})  applies to \eqref{eqg}. Then we see that
$$
a=-k\frac{2p-1}{4(p-2)},\ b=3k\frac{2p-1}{2(p-1)},\ c=-k\frac{2p-1}{2(p-2)},\ d=-3k\frac{p+1}{4(2-3p+p^2)},
$$ and thus by the facts that  $k>0, a\ge 0$ and $2<p<5$, we can deduce a contradiction.
\end{proof}


\begin{lem}\label{constr2}
Let $2<p<5$ and $\lambda, \mu_{ij} >0$ for  $i,j=1,2$.
 If $\vec{u}\in {\bf H}$  is a minimizer of $I$ on $\mathcal{M}$, then
$\vec{u}$ is a non-trivial critical point of $I$.
\end{lem}
\begin{proof}
Suppose that $\vec{u}=(u_1,u_2)\in {\bf H}$  is a minimizer of $I$ on $\mathcal{M}$. Then, by Lemma \ref{constr1}, there exists $\omega\in \R$ such that
$$
I^\prime(\vec{u})=\omega G^\prime(\vec{u}),
$$
which can be written as, in a weak sense, 
\begin{equation}\label{b2}
\begin{cases}  &-(1-3\omega)\Delta u_1+(1-\omega)\lambda u_1+(1-3\omega)(\mu_{11}\phi_{u_1}-\mu_{12}\phi_{u_2})u_1\\
&\qquad +\Big(-1+(2p-1)\omega\Big)\frac{1}{2\pi}\int_0^{2\pi} |u_1+e^{i\theta}u_2|^{p-1}(u_1+e^{i\theta}u_2)d\theta =0,\\
  &  -(1-3\omega)\Delta u_2+(1-\omega)\lambda u_2+(1-3\omega)(\mu_{22}\phi_{u_2}-\mu_{12}\phi_{u_1})u_2\\
&\qquad +\Big(-1+(2p-1)\omega\Big)\frac{1}{2\pi}\int_0^{2\pi}  |u_2+e^{i\theta}u_1|^{p-1}(u_2+e^{i\theta}u_1)d\theta=0.
\end{cases}
\end{equation}
Then, since $\vec{u}\in \mathcal{M}$, $k=I(\vec{u})$, $I^\prime(\vec{u})\vec{u}=\omega G^\prime(\vec{u})\vec{u}$ and the Pohozaev identity(see Proposition \ref{pz}) applied to \eqref{b2}, we have
\begin{align*}
\begin{cases}\frac12 a+\frac12 b+\frac{1}{4} c-\frac{1}{p+1} d=k,\\
\frac32a+\frac12 b+\frac34 c-\frac{2p-1}{p+1}d=0,\\
(1-3\omega)a+(1-\omega)b+(1-3\omega)c+\Big(-1+(2p-1)\omega\Big)d=0,\\
(1-3\omega)\frac12 a+(1-\omega)\frac32 b+(1-3\omega)\frac54c-\Big(1-(2p-1)\omega\Big)\frac{3}{p+1}d=0,
\end{cases}
\end{align*}
where   $a,b,c,d$ are given  in \eqref{nota1}. Then   since $b,d>0$ and $2<p<5$, we can prove $\omega=0$ (see \cite[Theorem 3.2]{R}).
\end{proof}

\begin{lem}\label{es1} It holds that
$$
\|\phi_u-\phi_v\|_{D^{1,2}}\le C\|u-v\|_{L^3}\|u+v\|_{L^2},
$$
where $C>0$ is a constant.
\end{lem}
\begin{proof}
Since $-\Delta(\phi_u-\phi_v)=u^2-v^2$, we have
\begin{align*}
\int_{\R^3}|\nabla(\phi_u-\phi_v)|^2dx&\le \int_{\R^3}(u-v)(u+v)(\phi_u-\phi_v)dx\le \|u-v\|_{L^3}\|u+v\|_{L^2}\|\phi_u-\phi_v\|_{L^6}\\
&\le C\|u-v\|_{L^3}\|u+v\|_{L^2}\|\phi_u-\phi_v\|_{D^{1,2}}.
\end{align*}
\end{proof}

\begin{prop}\label{proth}
Let $2<p<5$ and $\lambda, \mu_{ij} >0$ for  $i,j=1,2$.
Then $I$ has a global minimum $\vec{u}$ on $\mathcal{M}$. Moreover, $\vec{u}$ is a non-trivial, non-negative   critical point of $I$.
\end{prop}
\begin{proof}
Let $ \vec{u}_n =(u_{1,n},u_{2,n})\in \mathcal{M}$ such that $I(\vec{u}_n)\rightarrow \inf_{\mathcal{M}}I$. We claim that $\{\vec{u}_n\}$ is bounded in $H^1(\R^3)\times H^1(\R^3)$. Denote
\begin{equation}\label{je2}
\begin{aligned}
&a_n=\int_{\R^3}  |\nabla u_{1,n}|^2+|\nabla u_{2,n}|^2dx, \ \ b_n=\int_{\R^3}\lambda u_{1,n}^2+\lambda u_{2,n}^2dx,  \\
&c_n =\int_{\R^3}\mu_{11}u_{1,n}^2 \phi_{u_{1,n}}+\mu_{22}u_{2,n}^2 \phi_{u_{2,n}}^2 -2 \mu_{12} u_{1,n}^2 \phi_{u_{2,n}} dx, \\ &d_n=\int_{\R^3}\bigg[ \frac{1}{2\pi}\int_0^{2\pi} \Big(u_{1,n}^2+2u_{1,n} u_{2,n} \cos \theta+u_{2,n}^2\Big)^\frac{p+1}{2}d\theta\bigg] dx.
\end{aligned}
\end{equation}
From \eqref{b3}, we get that $b_n$ and $d_n$ are bounded, and hence by \eqref{err1},
  $c_n$ is bounded. Thus, since $ \vec{u}_n =(u_{1,n},u_{2,n})\in \mathcal{M}$ and $b_n, c_n, d_n$ are bounded, we have that $a_n$ is bounded.

  We may assume that  $\vec{u}_n\rightharpoonup \vec{u}=(u_1,u_2)$ in $H^1(\R^3)\times H^1(\R^3)$ and 
$\vec{u}_n\rightarrow  \vec{u}$ in $L^q(\R^3)\times L^q(\R^3)$, where $2<q<6$.
Define
\begin{align*}
&a =\int_{\R^3}  |\nabla u_1|^2+|\nabla u_2|^2dx, \ \ b=\int_{\R^3}\lambda u_1^2+\lambda u_2^2dx,  \\
&c =\int_{\R^3}\mu_{11}u_1^2 \phi_{u_1}+\mu_{22}u_2^2 \phi_{u_2} -2 \mu_{12} u_1^2 \phi_{u_2} dx, \ \ d=\int_{\R^3}\bigg[ \frac{1}{2\pi}\int_0^{2\pi} \Big(u_1^2+2u_1 u_2 \cos \theta+u_2^2\Big)^\frac{p+1}{2}d\theta\bigg] dx\\
&\bar{a}=\lim_{n\rightarrow \infty} a_n, \ \ \bar{b}=\lim_{n\rightarrow \infty} b_n, \ \ \bar{c}=\lim_{n\rightarrow \infty} c_n, \ \ \bar{d}=\lim_{n\rightarrow \infty} d_n,
\end{align*}
where $a_n, b_n, c_n$ and $d_n$ are given in \eqref{je2}.
Then, by the compactness of the embedding $H_r^1\rightarrow L^q$ and Lemma \ref{es1},  where $2<q<6$, we have $c=\bar{c}$ and $d=\bar{d}$. We claim that $\vec{u}_n\rightarrow \vec{u}$ in $H^1(\R^3)\times H^1(\R^3)$ and $\vec{u}\in \mathcal{M}$. To prove the claim, it suffices to  show that $a+b=\bar{a}+\bar{b}$. From the weak convergence, we have  $a\le \bar{a}$ and  $b\le \bar{b}$. Suppose that $\int_{\R^3}  |\nabla u_1|^2+\lambda u_1^2dx<\lim_{n\rightarrow\infty}\int_{\R^3}  |\nabla u_{1,n}|^2+\lambda u_{1,n}^2dx$. This implies that $a+b<\bar{a}+\bar{b}$. Since $I(\vec{u}_n)\rightarrow \inf_{\mathcal{M}}I$ and $G(\vec{u}_n)=0$, we have
\begin{equation}\label{b4}
\begin{cases} \frac12\bar{a}+\frac12\bar{b}+\frac14\bar{c}-\frac{1}{p+1}\bar{d}=\inf_{\mathcal{M}} I,\\
\frac32\bar{a}+\frac12\bar{b}+\frac34\bar{c}-\frac{2p-1}{p+1} \bar{d}=0.\end{cases}
\end{equation}
By \eqref{b1} and the second equation of \eqref{b4}, $0<\frac{p+1}{2p-1}\Big(\frac32\bar{a}+\frac12\bar{b}\Big)=-\frac{3(p+1)}{4(2p-1)}\bar{c}+\bar{d}=-\frac{3(p+1)}{4(2p-1)}c+d$, which implies that $\vec{u}\neq \vec{0}$. Then, by \eqref{je1}, we have $\bar{a}\ge a>0, \bar{b}\ge b>0$ and $\bar{d}=d>0$.
Define 
\begin{align*}
\bar{h}(t)=\frac12\bar{a}t^3+\frac12\bar{b}t+\frac14\bar{c}t^3-\frac{1}{p+1} \bar{d}t^{2p-1}\ \ \mbox{ and } \ \ h(t)=\frac12at^3+\frac12bt+\frac14ct^3-\frac{1}{p+1} dt^{2p-1}.
\end{align*} 
By Lemma \ref{uniqct}, both functions have a unique critical point, corresponding to their maxima. From \eqref{b4}, the maximum of $\bar{h}$ is equal to $\inf_{\mathcal{M}} I$ and is achieved at $t=1$. We observe that since  $c=\bar{c}$, $d=\bar{d}$, $a\le \bar{a}$, $b\le \bar{b}$ and $a+b<\bar{a}+\bar{b}$, $h(t)<\bar{h}(t)$ for all $t>0$. Let $t_0>0$ be the point such that the maximum of $h$ is achieved. Then $h^\prime(t_0)=0$ and $h(t_0)<\max \bar{h}=\inf_{\mathcal{M}} I.$ Define $\vec{v}_0(x) =(t_0^2u_1(t_0x),t_0^2u_2(t_0x)).$ Then we have
\begin{align*}
I(\vec{v}_0)&=\frac12at_0^3+\frac12bt_0+\frac14ct_0^3-\frac{1}{p+1} dt_0^{2p-1}=h(t_0)<\inf_{\mathcal{M}} I,\\
G(\vec{v}_0)&=\frac32at_0^3+\frac12bt_0+\frac34ct_0^3-\frac{2p-1}{p+1}d t_0^{2p-1}=t_0h^\prime(t_0)=0.
\end{align*}
Thus, $\vec{v}_0\in \mathcal{M}$ and $I(\vec{v}_0)<\inf_{\mathcal{M}}I$, which is a contradiction. Thus, we have $\vec{u}_n\rightarrow \vec{u}$ in $H^1(\R^3)\times H^1(\R^3)$ and $\vec{u}\in \mathcal{M}$.

To prove the existence of a non-negative  non-trivial critical point of $I$, we claim that 
\begin{equation}\label{posti2}
\begin{aligned}
  \int_0^{2\pi} \Big(a^2+2a b \cos \theta+b^2\Big)^\frac{p+1}{2}d\theta=\int_0^{2\pi} \Big(a^2+2|a| |b| \cos \theta+b^2\Big)^\frac{p+1}{2}d\theta,
\end{aligned}
\end{equation}
where $ a,b\in \R$. Indeed,   since  $t^2-2t\cos\theta+1>0, t^2+2t\cos\theta+1>0$ for $|t|<1$, $0\le\theta\le 2\pi$ and
$$
 \int_0^{2\pi}(2t\cos\theta+t^2)^nd\theta=\int_0^{2\pi}(-2t\cos\theta+t^2)^nd\theta,
$$
where $n\in \N$ and   $|t|<1$, we have
\begin{align*}
\int_0^{2\pi}(1+2t\cos\theta+t^2)^\frac{p+1}{2}d\theta&=\sum_{n=0}^\infty\frac{h^{(n)}(0)}{n!}\int_0^{2\pi}(2t\cos\theta+t^2)^nd\theta\\
&=\sum_{n=0}^\infty\frac{h^{(n)}(0)}{n!}\int_0^{2\pi}(-2t\cos\theta+t^2)^nd\theta\\
&=\int_0^{2\pi}(1-2t\cos\theta+t^2)^\frac{p+1}{2}d\theta,
\end{align*}
where  $|t|<1$ and $h(r)=(1+r)^{\frac{p+1}{2}}$, and thus   we deduce the claim \eqref{posti2}.
By \eqref{posti2}, if $\vec{u}=(u_1,u_2)$ is a minimizer of $I$ restricted to $\mathcal{M}$, then so is $ (|u_1|,|u_2|)$. Then, by Lemma \ref{constr2}, we see that $ (|u_1|,|u_2|)$ is a non-negative  non-trivial critical point of $I$.
\end{proof}

\subsection{Proof of Theorem \ref{th3}} 
We prove the existence part in subsection \ref{p25}.  By Proposition \ref{proth}, we see that for $2<p<5$, there exists a  non-trivial, non-negative ground state solution  $\vec{w}$ of \eqref{gme}. Moreover, note that by Lemma \ref{constr1}, $\vec{w}$ is a minimizer for $I$ in $\mathcal{M}$, $\mathcal{M}$ is   smooth in a neighborhood of $\vec{w}$ and $\mathcal{M}$ is  of codimension $1$ in a neighborhood of $\vec{w}$, where $\mathcal{M}$ is given in \eqref{mmf}. Hence,  a non-negative ground state solution $\vec{w}$ for \eqref{gme} has Morse index less than equal to $1$.

For $\mu_{ij}>0$, let  $W$ be a ground state solution of
\begin{equation}\label{sequa}
 -\Delta u+ \lambda  u + ( {\mu_{11} - \mu_{12}}) \phi_{u} u= \frac{c_p}{2}|u|^{p-1} u,
\end{equation}
where $c_p$ is given in \eqref{cp}. The existence of a ground state solution to \eqref{sequa} was proved in \cite{V} when $\mu_{11}-\mu_{12}<0$ and \cite[Theorem 3.2]{R}  when $\mu_{11}-\mu_{12}>0$. Clearly, there exists a ground state to \eqref{sequa} when $\mu_{11}=\mu_{12}$.
We note that     if $\mu_{11}=\mu_{22}$, $(W,W)$ is a solution of \eqref{gme}, and thus $(W,W)\in \mathcal{M}$.

First, we assume that $\frac13 (-2 + \sqrt{73})\le p<5$ and $\mu_{ij}>0$, where $i,j=1,2$. Then, by Proposition \ref{msi1},  a ground state solution of \eqref{gme} $\vec{w}$      is positive, where we used the fact that a non-negative ground state solution $\vec{w}$ for \eqref{gme} has Morse index less than equal to $1$.

Next, we assume that $2<p<5$ and $\mu_{11}=\mu_{22}$. Then, by Proposition \ref{egc1} and the fact that $I(W,W)\ge I(\vec{w})$, we see that a ground state solution  $\vec{w}$  is positive. Moreover, by Proposition \ref{rig1}, we deduce that $\vec{w}=(\tilde{W},\tilde{W})$, where $\tilde{W}$ is a positive  ground state solution of \eqref{sequa}.
 $\Box$

\section{Power potential case with $1<p\le2$ and $\det(\mu_{ij}) <0$}\label{negadet}  
In this subsection, we prove Theorem \ref{th5}, the existence of a positive solution of \eqref{gme} for $1<p\le2$ and large $\mu_{12}>0$.
This can be considered as  an extreme case of $\det(\mu_{ij}) <0$.

\subsection{Proof of Theorem \ref{th5}}
Let $v_i(x)=\sqrt{\mu_{12}}u_i(x)$ for $i=1,2$. Then
\begin{align*}
I_+(\vec{u})&=\frac12\int_{\R^3}|\nabla u_1|^2+|\nabla u_2|^2+\lambda u_1^2+\lambda u_2^2dx\\
&\quad +\frac14\int_{\R^3} \mu_{11}(u_1)_+^2 \phi_{(u_1)_+} + \mu_{22}(u_2)_+^2 \phi_{(u_2)_+}-2\mu_{12} (u_1)_+^2 \phi_{(u_2)_+} dx\\
&\qquad - \frac{1}{p+1} \frac{1}{2\pi}\int_{\R^3}\int_0^{2\pi} \Big((u_1)_+^2+2(u_1)_+ (u_2)_+ \cos \theta+(u_2)_+^2\Big)^\frac{p+1}{2}d\theta dx\\
&=\mu_{12}^{-1}\bigg[\frac{1}{2}\int_{\R^3}|\nabla v_1|^2+|\nabla v_2|^2+\lambda v_1^2+\lambda v_2^2dx\\
&\qquad +\frac{1}{4\mu_{12}}\int_{\R^3} \mu_{11}(v_1)_+^2 \phi_{(v_1)_+} + \mu_{22}(v_2)_+^2 \phi_{(v_2)_+}-2\mu_{12} (v_1)_+^2 \phi_{(v_2)_+} dx\\
&\qquad - \frac{1}{(p+1)\mu_{12}^{\frac{p-1}{2}}} \frac{1}{2\pi}\int_{\R^3}\int_0^{2\pi} \Big((v_1)_+^2+2(v_1)_+ (v_2)_+ \cos \theta+(v_2)_+^2\Big)^\frac{p+1}{2}d\theta dx\bigg]\\
&=\mu_{12}^{-1} \Big(I_0(\vec{v})+H_{\mu_{12}}(\vec{v})\Big),
\end{align*}
where $\vec{u}=(u_1,u_2)$, $\vec{v}=(v_1,v_2)$, 
$$
I_0(\vec{v})=\frac{1}{2}\int_{\R^3}|\nabla v_1|^2+|\nabla v_2|^2+\lambda v_1^2+\lambda v_2^2dx-\frac12\int_{\R^3} (v_1)_+^2 \phi_{(v_2)_+} dx
$$
and
\begin{align*}
H_{\mu_{12}}(\vec{v})&=\frac{1}{4\mu_{12}}\int_{\R^3} \mu_{11}(v_1)_+^2 \phi_{(v_1)_+} + \mu_{22}(v_2)_+^2 \phi_{(v_2)_+}dx\\
&\qquad - \frac{1}{2\pi(p+1)\mu_{12}^{\frac{p-1}{2}}}  \int_{\R^3}\int_0^{2\pi}  \Big((v_1)_+^2+2(v_1)_+ (v_2)_+ \cos \theta+(v_2)_+^2\Big)^\frac{p+1}{2}d\theta dx.
\end{align*}
By applying the result in \cite[Theorem 1]{1JS}, we try to find a critical point $\vec{v}$ of $I_0+H_{\mu_{12}}$.  
Indeed, we observe that by \eqref{ibp}, Lemma \ref{es1}, the compact embeddings $H^1_r(\R^3)\subset L^q(\R^3)$, where $2<q<6$ and the fact that
\begin{align*}
H_{\mu_{12}}^\prime(\vec{v})\vec{\psi}&=\frac{1}{\mu_{12}}\int_{\R^3} \mu_{11}(v_1)_+ \psi_1 \phi_{(v_1)_+} + \mu_{22}(v_2)_+\psi_2 \phi_{(v_2)_+}dx\\
& -  \frac{1}{2\pi \mu_{12}^{\frac{p-1}{2}}}\int_{\R^3} \int_0^{2\pi} \Big((v_1)_+^2+2(v_1)_+ (v_2)_+ \cos \theta+(v_2)_+^2\Big)^\frac{p-1}{2}\\
&\qquad \times \left(\Big((v_1)_++1_{\{v_1>0\}}(v_2)_+\cos\theta\Big)\psi_1+\Big((v_2)_++1_{\{v_2>0\}}(v_1)_+\cos\theta\Big)\psi_2\right) d\theta dx,
\end{align*}
where $\vec{\psi}=(\psi_1,\psi_2)$, we see that
 $H_{\mu_{12}}:{\bf H}\rightarrow \R$ is a $C^1$ functional such that $H_{\mu_{12}}:{\bf H}\rightarrow \R$,   $H_{\mu_{12}}^\prime:{\bf H}\rightarrow {\bf H}^{-1}$ are compact, and  for any $M>0$,
$$
\lim_{\mu_{12}\rightarrow \infty}\sup_{\|\vec{v}\|_{\bf{H}} \le M}|H_{\mu_{12}}(\vec{v})|=\lim_{\mu_{12}\rightarrow \infty}\sup_{\|\vec{v}\|_{\bf{H}} \le M}\|H_{\mu_{12}}^\prime(\vec{v})\|_{{\bf H}^{-1}}=0.
$$ 
Also, note that   $I_0$ satisfies the Palais-Smale condition and the following properties:  
\begin{enumerate}
\item there exist $c, r>0$ such that if $\|\vec{u}\|_{\bf{H}}=r$, then $I_0(\vec{u})\ge c$, and there exists $\vec{w}\in  {\bf H}$ such that $\|\vec{w}\|_{\bf{H}}>r$ and $I_0(\vec{w})<0$;
\item $$\min_{\gamma \in \Gamma}\max_{s\in[0,1]}I_0(\gamma(s))=\inf\{I_0(\vec{v}) \ | \ \vec{v}\in {\bf H}\setminus \{\vec{0}\}, I_0^\prime(\vec{v})=0\},$$
where $\Gamma=\{\gamma\in C([0,1],{\bf H})\ | \ \gamma(0)=0, \gamma(1)=\vec{w}\}$;
\item there exists a curve $\gamma_0(s)\in \Gamma$ passing throungh $\vec{v}_0$ at $s=s_0$ and satisfying
$$
I_0(\vec{v}_0)>I_0(\gamma_0(s)) \mbox{ for all } s\neq s_0.
$$
\end{enumerate}
Indeed, using  the compact embeddings $H^1_r(\R^3)\subset L^q(\R^3)$, where $2<q<6$, we can prove that the functional $I_0$ satisfies the Palais-Smale condition, that is, for any   sequence  $\{\vec{u}_n \}\subset {\bf H}$  satisfying
$$|I_0(\vec{u}_n)|\le C \mbox{  uniformly in } n, \mbox{ and }
\|I_0^\prime(\vec{u}_n)\|_{{\bf H}^{-1}}\rightarrow 0 \mbox{ as } n\rightarrow \infty,$$
 $\{\vec{u}_n \}$ has a strong convergence subsequence in $H^1(\R^3)\times H^1(\R^3)$.

For $w\in C_c^\infty(\R^3)\setminus\{0\}$, $w>0$ and for large $t>0$,
$$
I_0(tw,tw)=t^2\int_{\R^3}|\nabla w|^2 +\lambda w^2dx-\frac{t^4}{2}\int_{\R^3} w^2 \phi_{w} dx<0
$$
and for some $C_1>0$,
\begin{align*}
I_0(\vec{u})&\ge \frac12(\|u_1\|_{\lambda}^2+\|u_2\|_{\lambda}^2)-\frac12\|\phi_{u_2}\|_{L^6}\|u_1\|_{L^\frac{12}{5}}^2\ge \frac12(\|u_1\|_{\lambda}^2+\|u_2\|_{\lambda}^2)-C_1\| u_2 \|_{\lambda}^2\|u_1\|_{\lambda}^2\\
&\ge \frac12\Big[ \|u_1\|_{\lambda}^2(1-C_1\|u_1\|_{\lambda}^2)+ \|u_2\|_{\lambda}^2(1-C_1\|u_2\|_{\lambda}^2) \Big],
\end{align*}
which imply that there exist $c,r>0$ such that if $\|\vec{u}\|_{{\bf H}}=r$, then $I_0(\vec{u})\ge c>0$ and there exists $(w,w)\in {\bf H}$ such that $\|(w,w)\|_{\bf{H}}>r$ and $I_0(w,w)<0$. This proves $(1)$. By   the mountain pass theorem \cite{AR}, we see that
 there exists a critical point $\vec{v}_0\in {\bf H}\setminus\{0\}$ of $I_0$ such that 
$$
I_0(\vec{v}_0)=\inf_{\gamma\in \Gamma} \max_{s\in[0,1]}I_0(\gamma(s))>0,
$$
where $\Gamma=\{\gamma\in C([0,1],{\bf H})\ |\ \gamma(0)=0, \gamma(1)=(w,w)\}$.

%

On the other hand, let $\vec{v}\in {\bf H}\setminus \{0\}$ and  $I_0^\prime(\vec{v})=0.$ By a strong maximum principle and the fact that  $\vec{v}=(v_1,v_2)$ satisfies
\begin{equation*}
\begin{cases} -\Delta v_1+\lambda v_1 - \phi_{(v_2)_+} (v_1)_+=0\ \ \mbox{ in }\ \  \R^3,
\\  -\Delta v_2+\lambda v_2 - \phi_{(v_1)_+} (v_2)_+=0 \ \ \mbox{ in }\ \  \R^3,\end{cases}
\end{equation*}
 we see that $\vec{v}$ is  positive. Then, by Lemma \ref{betaunq}, $\vec{v}=(V,V)$, where $V\in H_r^1$ satisfies $V> 0$ and $-\Delta V+\lambda V-\phi_V V=0$.  Moreover, from \cite[Lemma 9]{Le}, we see that $V$ is a unique radial, positive solution of $-\Delta w+\lambda w-\phi_w w=0$.  Thus, we see that $\vec{v}_0=(V,V)$,
 $$
I_0(\vec{v}_0)=\inf\{I_0(\vec{u})\ |\ \vec{u}\in {\bf H}\setminus \{0\}, I_0^\prime(\vec{u})=0\}
$$
and  for all $ s\neq 1$,
$$
I_0(\vec{v}_0)>I_0(\gamma_0(s))=s^2\int_{\R^3}|\nabla V|^2 +\lambda V^2dx-\frac{s^4}{2}\int_{\R^3} V^2 \phi_{V} dx  =\left(s^2-\frac{s^4}{2}\right)\int_{\R^3}|\nabla V|^2 +\lambda V^2dx,
$$
where $\gamma_0(s)=(sV,sV)$. These prove $(2)$ and $(3)$.
Thus, by \cite[Theorem 1]{1JS}, we can find a critical point $\vec{v}_{\mu_{12}}$ of $I_0 +H_{\mu_{12}} $ such that $\vec{v}_{\mu_{12}}\rightarrow \vec{v}_0=(V,V)$ in $H^1(\R^3)\times H^1(\R^3)$ as $\mu_{12}\rightarrow \infty$.
 Then from   this, \eqref{ibp}, a strong maximum principle and the fact that $\vec{v}_{\mu_{12}}=(v_{1,\mu_{12}},v_{2,\mu_{12}})$ satisfies
\begin{equation*}
\begin{cases} -\Delta v_1+\lambda v_1+\frac{1}{\mu_{12}}(\mu_{11}\phi_{(v_1)_+}-\mu_{12}\phi_{(v_2)_+})(v_1)_+\\
\ \ =\frac{1}{2\pi\mu_{12}^{\frac{p-1}{2}}}\int_0^{2\pi}\Big((v_1)_+^2+(v_2)_+^2+2(v_1)_+(v_2)_+\cos\theta\Big)^{\frac{p-1}{2}}\Big((v_1)_++1_{\{v_1>0\}}(v_2)_+\cos\theta\Big)d\theta \   \mbox{ in }\   \R^3,
\\  -\Delta v_2+\lambda v_2+\frac{1}{\mu_{12}}(\mu_{22}\phi_{(v_2)_+}-\mu_{12}\phi_{(v_1)_+})(v_2)_+\\
\ \ =\frac{1}{2\pi\mu_{12}^{\frac{p-1}{2}}}\int_0^{2\pi}\Big((v_1)_+^2+(v_2)_+^2+2(v_1)_+(v_2)_+\cos\theta\Big)^{\frac{p-1}{2}}\Big((v_2)_++1_{\{v_2>0\}}(v_1)_+\cos\theta\Big)d\theta \   \mbox{ in }\   \R^3. \end{cases}
\end{equation*}
  we see that for large $\mu_{12}$, $\vec{v}_{\mu_{12}}$ is positive. Thus, we deduce that for large $\mu_{12}$, $(\mu_{12})^{-\frac12}\vec{v}_{\mu_{12}}$ is a positive critical point of $I$.

\section{Power potential case with $1 < p < 2$ and $\det(\mu_{ij}) > 0$}\label{posidet}
In this section, we prove Theorem \ref{th31} and Theorem \ref{th4}, which give the non-existence result to \eqref{gme} when $\mu_{11}$ and $\mu_{22}$ are large, and the existence result of two positive solutions to \eqref{gme} when $\det(\mu_{ij}) > 0$ and $\mu_{ij}$ is small, respectively.
\subsection{ Triviality of $\vec{u}$ when $\mu_{11}$ and $\mu_{22}$ are large} We first prove that any solution $\vec{u}$ to \eqref{gme} is trivial when $1<p\le 2$, $\lambda\ge2$, $\mu_{ij}>0$, $\mu_{11}>4$ and $(\mu_{11}-4)(\mu_{22}-4)>\mu_{12}^2.$

\noindent{\bf Proof of Theorem \ref{th31}.}
Suppose that $\vec{u}=(u_1,u_2)\in H^1(\R^3)\times H^1(\R^3)$ is a solution of \eqref{gme}. Then we have
\begin{align*}
 I^\prime(\vec{u})\vec{u} &=\int_{\R^3}|\nabla u_1|^2+|\nabla u_2|^2+\lambda u_1^2+\lambda u_2^2dx+\int_{\R^3} \mu_{11}u_1^2 \phi_{u_1} + \mu_{22}u_2^2 \phi_{u_2}-2\mu_{12} u_1^2 \phi_{u_2} dx\\
&\qquad -  \frac{1}{2\pi}  \int_{\R^3}\int_0^{2\pi}\Big(u_1^2+2u_1 u_2 \cos \theta+u_2^2\Big)^\frac{p+1}{2} d\theta dx=0.
\end{align*}
We note that for $a,b\in \R$,
$$
 \frac{1}{2\pi}\int_0^{2\pi}\Big(a^2+2a b \cos \theta+b^2\Big)^\frac{p+1}{2} d\theta\le (|a| +|b| )^{p+1} \le 2^p(|a|^{p+1}+|b|^{p+1})
$$
and
$$
 4\int_{\R^3}|u|^3dx= 4\int_{\R^3}(-\Delta \phi_{u})|u|dx= 4\int_{\R^3} \nabla \phi_{u} \cdot \nabla |u|dx\le \int_{\R^3} |\nabla u|^2+4|\nabla \phi_{u}|^2dx.
$$
Then we see that if $\mu_{11}>4$ and $\Big(\mu_{11}-4\Big)\Big(\mu_{22}-4\Big)>\mu_{12}^2,$
\begin{align*}
0= I^\prime(\vec{u})\vec{u} &\ge \int_{\R^3}\lambda u_1^2+\lambda u_2^2dx+\int_{\R^3} \Big(\mu_{11}-4\Big)u_1^2 \phi_{u_1} + \Big(\mu_{22}-4\Big)u_2^2 \phi_{u_2}-2\mu_{12} u_1^2 \phi_{u_2} dx\\
&\quad +4\int_{\R^3}|u_1|^3+|u_2|^3dx-2^p\int_{\R^3}|u_1|^{p+1}+|u_2|^{p+1}dx\\
&\ge \int_{\R^3}\lambda u_1^2+\lambda u_2^2dx  +4\int_{\R^3}|u_1|^3+|u_2|^3dx-2^p\int_{\R^3}|u_1|^{p+1}+|u_2|^{p+1}dx.
\end{align*}
To complete the proof, we claim that if $\lambda\ge2$ and $1<p\le2$, then $h(t)\ge 0$ for $t\ge0$, where 
   $h:\R^+\rightarrow \R$ is given by $h(t)=\lambda +4t-2^p t^{p-1}$. We observe that 
$$
h^\prime(t)\begin{cases}<0 &\mbox{ if } t<  \frac12(p-1)^\frac{1}{2-p},\\
>0 &\mbox{ if } t>\frac12(p-1)^\frac{1}{2-p}, \\
=0 &\mbox{ if } t=\frac12(p-1)^\frac{1}{2-p},\end{cases}
$$
$$h(t)\ge 0 \ \ \mbox{ for }t\ge 0 \ \ \mbox{ if }\ \ \lambda\ge 2(p-1)^\frac{p-1}{2-p}-2(p-1)^\frac{1}{2-p}= 2   (p-1)^\frac{p-1}{2-p}(2-p),
$$
and for $1<p\le 2$, $ (p-1)^\frac{p-1}{2-p}(2-p)<  (p-1)^\frac{p-1}{2-p}\le 1$.
From these, we can prove the claim, and  thus,    $(u_1,u_2)=(0,0)$ if $\lambda \ge  2$ and $1<p\le 2$. $\Box$

\subsection{Construction of two positive solutions}\label{p12}
Following the idea in \cite{R}, we study  the existence of two positive solutions of \eqref{gme}   for $1<p<2$ and   $\mu_{11}\mu_{22}-\mu_{12}^2>0$. 

\begin{prop}\label{psc}
Let $\lambda, \mu_{ij} >0$ for  $i,j=1,2$. Assume $1<p<2$ and $\mu_{11}\mu_{22}-\mu_{12}^2>0$. Then we have 
\begin{enumerate}[(i)] 
\item $\inf_{\vec{u}\in {\bf H}} I(\vec{u})>-\infty$;\\
\item $I$ satisfies (PS) condition.
\end{enumerate} 
\end{prop}
\begin{proof}
We first claim that $\inf_{\vec{u}\in {\bf H}} I(\vec{u})>-\infty.$
We note that  for $a,b\in \R$,
$$
 \frac{1}{2\pi}   \int_0^{2\pi}\Big(a^2+2ab \cos \theta+b^2\Big)^\frac{p+1}{2} d\theta\le (|a|+|b|)^{p+1}\le 2^{p+1}(|a|^{p+1}+|b|^{p+1}),
$$
 and for some  $0<\tilde{\mu}_{11}<\mu_{11}$, $0<\tilde{\mu}_{22}<\mu_{22}$, $\mu_{12}<\tilde{\mu}_{12}$, we have $\tilde{\mu}_{11}\tilde{\mu}_{22}-\tilde{\mu}_{12}^2>0$.
  Then we have 
\begin{align*}
I(\vec{u})&\ge \frac12\int_{\R^3}|\nabla u_1|^2+|\nabla u_2|^2+\lambda u_1^2+\lambda u_2^2dx\\
&\quad +\frac14\int_{\R^3}(\mu_{11}-\tilde{\mu}_{11})u_1^2 \phi_{u_1} + (\mu_{22}-\tilde{\mu}_{22})u_2^2 \phi_{u_2}+ 2(\tilde{\mu}_{12}-\mu_{12})u_1^2 \phi_{u_2}dx\\
&\quad +\frac14\int_{\R^3} \tilde{\mu}_{11}u_1^2 \phi_{u_1} + \tilde{\mu}_{22}u_2^2 \phi_{u_2}-2\tilde{\mu}_{12} u_1^2 \phi_{u_2} dx - \frac{2^{p+1}}{p+1} \int_{\R^3}  |u_1|^{p+1}+|u_2|^{p+1} dx\\
&\ge  \frac12\int_{\R^3}|\nabla u_1|^2+|\nabla u_2|^2+\lambda u_1^2+\lambda u_2^2dx+\frac{\hat{\mu}}{4}\int_{\R^3}  u_1^2 \phi_{u_1} +  u_2^2 \phi_{u_2}+  u_1^2 \phi_{u_2} dx \\
&\quad - \frac{2^{p+1}}{p+1} \int_{\R^3}  |u_1|^{p+1}+|u_2|^{p+1} dx,
\end{align*}
where $\hat{\mu}\in\Big(0,\min\{\mu_{11}-\tilde{\mu}_{11}, \mu_{22}-\tilde{\mu}_{22},\tilde{\mu}_{12}-\mu_{12}\}\Big).$
From this and the fact that
$$
\sqrt{\frac{ \hat{\mu}}{8}}\int_{\R^3}|u|^3dx=\sqrt{\frac{\hat{\mu}}{8}}\int_{\R^3}(-\Delta \phi_{u})|u|dx=\sqrt{\frac{\hat{\mu}}{8}}\int_{\R^3} \nabla \phi_{u} \cdot \nabla |u|dx\le \int_{\R^3}\frac14 |\nabla u|^2+\frac{\hat{\mu}}{8}|\nabla \phi_{u}|^2dx,
$$ 
 we obtain 
\begin{equation}\label{xn1}
\begin{aligned}
I(\vec{u})&\ge  \frac14\int_{\R^3}|\nabla u_1|^2+|\nabla u_2|^2dx+ \frac{\lambda}{2}\int_{\R^3}  u_1^2+ u_2^2dx+\frac{ \hat{\mu}}{8}\int_{\R^3}u_1^2 \phi_{u_1} +u_2^2 \phi_{u_2}+ 2u_1^2 \phi_{u_2}dx \\
&\quad  + \sqrt{\frac{ \hat{\mu}}{8}}\int_{\R^3} |u_1|^3+ |u_2|^3dx- \frac{2^{p+1}}{p+1} \int_{\R^3}  |u_1|^{p+1}+|u_2|^{p+1} dx.
\end{aligned}
\end{equation}
Define
$$
h :\R^+\times \R^+\rightarrow \R,\ \  h(s,t)=\frac{\lambda}{4}(s^2+ t^2)+\sqrt{\frac{ \hat{\mu}}{8}}(s^3+  t^3)-\frac{2^{p+1}}{p+1}(s^{p+1}+t^{p+1}),
$$
where $p\in (1,2)$. We claim that  
\begin{equation}\label{xn2}
 h \mbox{ is positive for } \{(s,t)\in \R^+\times \R^+\ |\ s^2+t^2<r_0 \mbox{ or } s^2+t^2>R_0\},
\end{equation}
 where $r_0, R_0>0$ are constants.
Indeed, note that   if $(s,t)\in  \{(s,t)\in \R^+\times \R^+\ |\ s^2+t^2<r_0\}$,
$
\frac{s^{p+1}+t^{p+1}}{s^2+t^2}<r_0^{\frac{p-1}{2}}
$
and  if $(s,t)\in  \{(s,t)\in \R^+\times \R^+ \ |\ s<t, \ s^2+t^2>R_0\},$
$
\frac{s^{p+1}+t^{p+1}}{s^3+t^3}\le \frac{2t^{p+1}}{t^3}<2\Big(\frac{2}{R_0}\Big)^\frac{2-p}{2}.
$ These imply \eqref{xn2}.

Define $m=\inf_{s>0,t>0}h(s,t)$. If $m\ge0$, we are done, beacuse $I(\vec{u})\ge0$, and thus we assume $m<0$. Then, by \eqref{xn1} and \eqref{xn2}, we have 
$$\{(s,t)\in \R^+\times \R^+\ | \ h(s,t)<0\}\subset  \{(s,t)\in \R^+\times \R^+\ |\ r_0<s^2+t^2<R_0\}$$
and
\begin{equation}\label{bddbelow}
\begin{aligned}
I(\vec{u})&\ge  \frac14\int_{\R^3}|\nabla u_1|^2+|\nabla u_2|^2dx+ \frac{\lambda}{4}\int_{\R^3} u_1^2+  u_2^2dx+\frac{\hat{\mu}}{8}\int_{\R^3}u_1^2 \phi_{u_1} +u_2^2 \phi_{u_2}+ 2u_1^2 \phi_{u_2}dx \\
&\quad  + \int_{\R^3}h(|u_1|,|u_2|) dx\\
&\ge  \frac14\int_{\R^3}|\nabla u_1|^2+|\nabla u_2|^2dx+ \frac{\lambda}{4}\int_{\R^3}  u_1^2+  u_2^2dx+\frac{\hat{\mu}}{8}\int_{\R^3}u_1^2 \phi_{u_1} +u_2^2 \phi_{u_2}+ 2u_1^2 \phi_{u_2}dx  \\
&\quad  + m \Big|\{x\in \R^3\ |\ h(|u_1(x)|,|u_2(x)|)<0\} \Big|\\
&\ge  \frac14\int_{\R^3}|\nabla u_1|^2+|\nabla u_2|^2dx+ \frac{\lambda}{4}\int_{\R^3}  u_1^2+  u_2^2dx+\frac{\hat{\mu}}{8}\int_{\R^3}u_1^2 \phi_{u_1} +u_2^2 \phi_{u_2}+ 2u_1^2 \phi_{u_2}dx +m|A|,
\end{aligned}
\end{equation}
where $A=\{x\in \R^3\ |\ r_0<u_1^2(x)+u_2^2(x) <R_0\}.$ Note that $A$ is spherically symmetric.

Suppose to the contrary   that there exists $\{\vec{u}_n\}\subset {\bf H}$ such that $I(\vec{u}_n)\rightarrow -\infty$. Clearly, $\|u_{1,n}\|_{\lambda}+\|u_{2,n}\|_{\lambda}\rightarrow \infty$ as $n\rightarrow \infty$. For each $\vec{u}_n=(u_{1,n},u_{2,n})$, define $A_n=\{x\in \R^3\ |\ r_0<u_{1,n}^2(x)+u_{2,n}^2(x) <R_0\}$ and $\rho_n=\sup\{|x|\ |\ x\in A_n\}$. By \eqref{bddbelow}  and the fact that  $I(\vec{u}_n)<0$, we have
$$
|m||A_n|\ge \frac14( \|u_{1,n}\|_{\lambda}^2 +\|u_{2,n}\|_{\lambda}^2),
$$
which implies that $|A_n|\rightarrow \infty$ as $n\rightarrow \infty$.
Since
$$|u(x)|\le C|x|^{-1}\|u\|_{H^1} \mbox{ for } u\in H_r^1,$$
where $C>0$ (see \cite{S}), we see that for $x\in \R^3$ with $|x|=\rho_n$, $u_{1,n}^2(x)+u_{2,n}^2(x)=r_0$ and
$$
r_0=u_{1,n}^2(x)+u_{2,n}^2(x)\le C^2\rho_n^{-2}(\|u_{1,n}\|_{H^1}^2+\|u_{2,n}\|_{H^1}^2)\le C_1\rho_n^{-2}|m||A_n|,
$$
which implies that $\rho_n\le C_2|A_n|^\frac12,$ where $C_1, C_2>0$.

On the other hand, by \eqref{bddbelow}  and the fact that  $I(\vec{u}_n)<0$, we have
$$
\frac{\hat{\mu}}{8}\int_{\R^3} u_{1,n}^2 \phi_{u_{1,n}} +u_{2,n}^2 \phi_{u_{2,n}}+ 2u_{1,n}^2 \phi_{u_{2,n}}dx \le |m||A|.
$$
Then from this and the facts that 
$
|x-y|\le 2\rho_n \mbox{ for } x,y \in A_n
$
and
\begin{align*}
 &\int_{\R^3} u_{1,n}^2 \phi_{u_{1,n}} +u_{2,n}^2 \phi_{u_{2,n}}+ 2u_{1,n}^2 \phi_{u_{2,n}}dx\\
&= \int_{\R^3}\int_{\R^3} \frac{u_{1,n}^2(x)u_{1,n}^2(y)+u_{2,n}^2(x)u_{2,n}^2(y)+2u_{1,n}^2(x)u_{2,n}^2(y)}{|x-y|} dxdy\\
&=\int_{\R^3}\int_{\R^3} \frac{\Big(u_{1,n}^2(x)+u_{2,n}^2(x)\Big)\Big(u_{1,n}^2(y)+u_{2,n}^2(y)\Big)}{|x-y|} dxdy,
\end{align*}
we have
\begin{align*}
\frac{8}{\hat{\mu}}|m||A_n|&\ge \int_{\R^3}\int_{\R^3} \frac{\Big(u_{1,n}^2(x)+u_{2,n}^2(x)\Big)\Big(u_{1,n}^2(y)+u_{2,n}^2(y)\Big)}{|x-y|} dxdy\\
&\ge \int_{A_n}\int_{A_n} \frac{\Big(u_{1,n}^2(x)+u_{2,n}^2(x)\Big)\Big(u_{1,n}^2(y)+u_{2,n}^2(y)\Big)}{|x-y|} dxdy\ge r_0^2\frac{|A_n|^2}{2\rho_n}.
\end{align*}
This implies that $|A_n|\le C_3\rho_n$, where $C_3>0$, which is a contradiction to the facts that $|A_n|\rightarrow \infty$ as $n\rightarrow \infty$ and $\rho_n\le C_2|A_n|^\frac12.$ Thus, we have $\inf_{\vec{u}\in {\bf H}} I(\vec{u})>-\infty.$

Next, we prove that $I$ satisfies  (PS) condition. Let $\vec{u}_n=(u_{1,n},u_{2,n})$ be a sequence in  ${\bf H}$ such that $I^\prime(\vec{u}_n)\rightarrow 0$ as $n\rightarrow \infty$. Then we have\begin{align*}
\|u_{1,n}\|_{\lambda}+\|u_{2,n}\|_{\lambda}&\ge I^\prime(\vec{u}_n)\vec{u}_n\\
&=\int_{\R^3}|\nabla u_1|^2+|\nabla u_2|^2+\lambda u_1^2+\lambda u_2^2dx+\int_{\R^3} \mu_{11}u_1^2 \phi_{u_1} + \mu_{22}u_2^2 \phi_{u_2}-2\mu_{12} u_1^2 \phi_{u_2} dx\\
&\qquad -  \frac{1}{2\pi}  \int_{\R^3}\int_0^{2\pi}\Big(u_1^2+2u_1 u_2 \cos \theta+u_2^2\Big)^\frac{p+1}{2} d\theta dx.
\end{align*}
Following the first part of the proof \eqref{xn1}, we deduce that for some $\bar{\mu}>0$,
\begin{equation}\label{engb1}
\begin{aligned}
\|u_{1,n}\|_{\lambda}+\|u_{2,n}\|_{\lambda}&\ge \frac12\int_{\R^3}|\nabla u_{1,n}|^2+|\nabla u_{2,n}|^2+\lambda u_{1,n}^2+\lambda u_{2,n}^2dx\\
&\quad +\frac{\bar{\mu}}{2}\int_{\R^3}  u_{1,n}^2 \phi_{u_{1,n}} +  u_{2,n}^2 \phi_{u_{2,n}}+  2u_{1,n}^2 \phi_{u_{2,n}} dx 
  +\int_{\R^3}g(|u_{1,n}|,|u_{2,n}|)dx,
\end{aligned}
\end{equation}
where $g(s,t)=\frac12\lambda (s^2+  t^2)+\sqrt{\bar{\mu}}(s^3+t^3)-2^{p+1}(s^{p+1}+t^{p+1})$.
If $\|u_{1,n}\|_{\lambda}+\|u_{2,n}\|_{\lambda}$ is unbounded in $n$, then  by \eqref{engb1}, for large $n$, we have
\begin{align*}
&\frac13\int_{\R^3}|\nabla u_{1,n}|^2+|\nabla u_{2,n}|^2+\lambda u_{1,n}^2+\lambda u_{2,n}^2dx+\frac{\bar{\mu}}{2}\int_{\R^3}  u_{1,n}^2 \phi_{u_{1,n}} +  u_{2,n}^2 \phi_{u_{2,n}}+  2u_{1,n}^2 \phi_{u_{2,n}} dx \\
&\quad +\int_{\R^3}g(|u_{1,n}|,|u_{2,n}|)dx<0.
\end{align*}
Then, by the same arguments in \eqref{bddbelow} below, we deduce a contradiction.
Since  $\|u_{1,n}\|_{\lambda}+\|u_{2,n}\|_{\lambda}$ is bounded, we assume that $\vec{u}_n\rightharpoonup \vec{u}=(u_1,u_2)$ in $H^1(\R^3)\times H^1(\R^3)$ and $\vec{u}_n\rightharpoonup \vec{u}$ in $L^q(\R^3)\times L^q(\R^3)$, where $2<q<6$. From this and Lemma \ref{es1},
we see that
\begin{align*}
&\lim_{n\rightarrow \infty} I^\prime(\vec{u}_n)\vec{u}\\
&= \lim_{n\rightarrow \infty}\bigg[\int_{\R^3} \nabla u_{1,n}\cdot \nabla u_1 + \nabla u_{2,n}\cdot \nabla u_2+\lambda u_{1,n} u_1+\lambda u_{2,n}u_2dx\\
&\quad   + \int_{\R^3}\mu_{11}   u_{1,n} \phi_{u_{1,n}} u_1 + \mu_{22}    u_{2,n}  \phi_{u_{2,n}} u_2 - \mu_{12} u_{1,n}  \phi_{u_{2,n}} u_1 -\mu_{12}  u_{2,n} \phi_{u_{1,n}} u_2 dx\\
&   -  \frac{1}{2\pi}\int_{\R^3}\int_0^{2\pi} \Big((u_{1,n}+u_{2,n} \cos \theta ) u_1+ (u_{2,n}+u_{1,n} \cos \theta )u_2\Big)\\
&\quad \times \Big(u_{1,n}^2+2u_{1,n} u_{2,n} \cos \theta+u_{2,n}^2\Big)^\frac{p-1}{2} d\theta dx\bigg]\\
&=\int_{\R^3}|\nabla u_1|^2+|\nabla u_2|^2+\lambda u_1^2+\lambda u_2^2dx+\int_{\R^3} \mu_{11}u_1^2 \phi_{u_1} + \mu_{22}u_2^2 \phi_{u_2}-2\mu_{12} u_1^2 \phi_{u_2} dx\\
&\quad -  \frac{1}{2\pi}  \int_{\R^3}\int_0^{2\pi}\Big(u_1^2+2u_1 u_2 \cos \theta+u_2^2\Big)^\frac{p+1}{2} d\theta dx=0
\end{align*}
and
\begin{align*}
&\lim_{n\rightarrow \infty}I^\prime(\vec{u}_n)\vec{u}_n\\
&=\lim_{n\rightarrow \infty}\bigg[\int_{\R^3}|\nabla u_{1,n}|^2+|\nabla u_{2,n}|^2+\lambda u_{1,n}^2+\lambda u_{2,n}^2dx\\
&\quad +\int_{\R^3} \mu_{11}u_{1,n}^2 \phi_{u_{1,n}} + \mu_{22}u_{2,n}^2 \phi_{u_{2,n}}-2\mu_{12} u_{1,n}^2 \phi_{u_{2,n}} dx\\
&\quad -  \frac{1}{2\pi}  \int_{\R^3}\int_0^{2\pi}\Big(u_{1,n}^2+2u_{1,n} u_{2,n} \cos \theta+u_{2,n}^2\Big)^\frac{p+1}{2} d\theta dx\bigg]\\
&=\lim_{n\rightarrow \infty}\bigg[\int_{\R^3}|\nabla u_{1,n}|^2+|\nabla u_{2,n}|^2+\lambda u_{1,n}^2+\lambda u_{2,n}^2dx\bigg]+\int_{\R^3} \mu_{11}u_1^2 \phi_{u_1} + \mu_{22}u_2^2 \phi_{u_2}-2\mu_{12} u_1^2 \phi_{u_2} dx\\
&\quad -  \frac{1}{2\pi}  \int_{\R^3}\int_0^{2\pi}\Big(u_1^2+2u_1 u_2 \cos \theta+u_2^2\Big)^\frac{p+1}{2} d\theta dx=0.
\end{align*}
Thus,  we deduce that  $\vec{u}_n\rightarrow \vec{u}=(u_1,u_2)$ in $H^1(\R^3)\times H^1(\R^3)$.
\end{proof}

\begin{prop}\label{existwo}
Let $1<p<2$ and $\lambda, \mu_{ij} >0$ for  $i,j=1,2$. If  $\mu_{11}\mu_{22}-\mu_{12}^2>0$  and $\mu_{ij}$ is small enough, where $i,j=1,2$, then  there are at least two different   non-negative radial solutions $\vec{u}_{1,\mu_{ij}}$, $\vec{u}_{2,\mu_{ij}}$ of  \eqref{gme}, where $\vec{u}_{1,\mu_{ij}}$ is a non-negative minimizer of  \eqref{gme} with negative energy, and $\vec{u}_{2,\mu_{ij}}$ is a positive solution  of \eqref{gme} with positive energy.
\end{prop}
\begin{proof}
We observe that if $\mu_{ij}=0$ for $i,j=1,2$, $I$ is not bounded below, and thus  there exists $\mu_0>0$ such that if $0<\mu_{ij}<\mu_0$, then $\inf_{\vec{u}\in {\bf H}} I(\vec{u})<0$. By Proposition \ref{psc}, $I$ is bounded below and satisfies (PS) condition. Then, by the Ekeland variational principle \cite{E}, the infimum is achieved, that is, there exists $\vec{u}_{1,\mu_{ij}}\in {\bf H}\setminus \{(0,0)\}$ such that $\inf_{\vec{u}\in {\bf H}} I(\vec{u})= I(\vec{u}_{1,\mu_{ij}})<0$ for $0<\mu_{ij}<\mu_0$. Moreover, by \eqref{posti2}, we may assume that the minimizer $\vec{u}_{1,\mu_{ij}}$ is non-negative.

Next, in order to prove the existence of second positive solution to \eqref{gme},  we consider a limit problem 
\begin{equation}\label{limg} 
\begin{cases} -\Delta u_1+\lambda u_1 =\frac{1}{2\pi} \int_0^{2\pi}\Big((u_1)_+^2+(u_2)_+^2+2(u_1)_+(u_2)_+\cos\theta\Big)^{\frac{p-1}{2}} ((u_1)_++1_{\{u_1>0\}}(u_2)_+\cos\theta )d\theta,
\\  -\Delta u_2+\lambda u_2 =\frac{1}{2\pi}\int_0^{2\pi}\Big((u_1)_+^2+(u_2)_+^2+2(u_1)_+(u_2)_+\cos\theta\Big)^{\frac{p-1}{2}} ((u_2)_++1_{\{u_2>0\}}(u_1)_+\cos\theta )d\theta \end{cases}
\end{equation}
and a corresponding energy functional
$$
\begin{aligned}
I_{+,0}(\vec{u})=&\frac12\int_{\R^3}|\nabla u_1|^2+|\nabla u_2|^2+\lambda u_1^2+\lambda u_2^2dx\\
&\quad - \frac{1}{p+1} \frac{1}{2\pi}\int_{\R^3}\int_0^{2\pi} \Big((u_1)_+^2+2(u_1)_+ (u_2)_+ \cos \theta+(u_2)_+^2\Big)^\frac{p+1}{2}d\theta dx.
\end{aligned}
$$
Then by \eqref{ibppa} and a maximum principle, $I_{+,0}$ is $C^1$ and any critical point of $I_{+,0}$ is a non-negative solution of \eqref{limg}. It is standard to show that
\begin{equation}\label{cond1}
I_{+,0} \mbox{ satisfies (PS) condition }
\end{equation}
because of the arguments in Proposition \ref{psc} (ii) and the fact that $I_{+,0}(\vec{u})-\frac{1}{p+1}I_{+,0}'(\vec{u})\vec{u}=(\frac12-\frac{1}{p+1})\|\vec{u}\|_{\bf{H}}^2$. Moreover, we see that there exist $c, r>0$ such that 
\begin{equation}\label{cond22}
  \mbox{if } \|\vec{u}\|_{\bf{H}}=r, \mbox{ then }I_{+,0}(\vec{u})\ge c, \mbox{ and there exists } \vec{w}\in  {\bf H} \mbox{ such that } \|\vec{w}\|_{\bf{H}}>r \mbox{ and } I_{+,0}(\vec{w})<0.
\end{equation}
 Then by the mountain pass theorem of  Ambrosetti and Rabinowitz \cite{AR}, we can prove that there exists a non-negative mountain pass solution $\vec{v}$ with Morse index for \eqref{limg} less than equal to $1$. Note that by Lemma \ref{semimorse}, the Morse index of semi-trivial solutions $(w,0)$ and $(0,w)$ for \eqref{limg} are larger than equal to $2$, where $w$ is a unique positive solution of 
\begin{equation}\label{schr1}
-\Delta w+\lambda w=w^p.
\end{equation} Thus, a non-negative mountain pass solution $\vec{v}$ is a positive  solution. Then by Proposition \ref{rig1}, we see that $\vec{v}=(U,U),$
 where $U$ is a unique positive solution of
\begin{equation}\label{lseq}
-\Delta U+\lambda U=\frac{c_p}{2} U^p
\end{equation} 
and $c_p$ is given in \eqref{cp} Moreover, 
\begin{equation}\label{cond12}
I_{+,0}(U,U)=(\frac12-\frac{1}{p+1})\|(U,U)\|_{\bf{H}}^2>0.
\end{equation}

We employ  the result in \cite[Theorem 1]{1JS} to find a critical point $\vec{u}_{2,\mu_{ij}}$ of $I_+$ such that $\vec{u}_{2,\mu_{ij}}$ converges to $(U,U)$ as $\mu_{ij}\rightarrow 0$. We write
\begin{align*}
I_+(\vec{u})&=I_{+,0}(\vec{u})+H_{\mu_{11},\mu_{22},\mu_{12}}(\vec{u}),
\end{align*}
where $I_+$ is given in \eqref{funpl} and 
$$
H_{\mu_{11},\mu_{22},\mu_{12}}(\vec{u})=\frac14\int_{\R^3} \mu_{11}(u_1)_+^2 \phi_{(u_1)_+} + \mu_{22}(u_2)_+^2 \phi_{(u_2)_+}-2\mu_{12} (u_1)_+^2 \phi_{(u_2)_+} dx.
$$
Observe that since
$J_{+,0}'(\vec{u})\vec{u}<-\rho$ if $\vec{u}\in {\bf{H}}\setminus \{(0,0)\}$ and $I_{+,0}'(\vec{u})\vec{u}=0$, where $J_{+,0}(\vec{u})= I_{+,0}'(\vec{u})\vec{u}$ and $\rho$ is a positive constant, we get
\begin{align*}
\inf\{ I_{+,0}(\vec{u}) \ |\ \vec{u}\in {\bf{H}}\setminus \{(0,0)\}, I_{+,0}'(\vec{u})\vec{u}=0\}&=\inf\{ I_{+,0}(\vec{u}) \ |\ \vec{u}\in {\bf{H}}\setminus \{(0,0)\}, I_{+,0}'(\vec{u})=0\}\\
&=\min\{ I_{+,0}(U,U),  I_{+,0}(w,0),  I_{+,0}(0,w)\},
\end{align*}
where $U$ is a positive solution of \eqref{lseq} and $w$ is a positive solution of \eqref{schr1}. Then by the fact that  the Morse index of semi-trivial solutions $(w,0)$ and $(0,w)$ for \eqref{limg} are larger than equal to $2$, we see that 
\begin{equation}\label{cond3}
 I_{+,0}(U,U)=\inf\{ I_{+,0}(\vec{u}) \ |\ \vec{u}\in {\bf{H}}\setminus \{(0,0)\}, I_{+,0}'(\vec{u})=0\}.
\end{equation}
Also, we see that for $s\neq1$,
\begin{equation}\label{cond4}
 I_{+,0}(U,U)> I_{+,0}(\gamma(s))=\left(s^2-\frac{2}{p+1}s^{p+1}\right)\int_{\R^3}|\nabla U|^2+\lambda U^2dx
\end{equation}
where $\gamma(s)=(sU,sU)$, and by the compact embeddings $H^1_r(\R^3)\subset L^q(\R^3)$, where $2<q<6$, $H_{\mu_{11},\mu_{22},\mu_{12}}:\bf{H}\rightarrow \R$ is a $C^1$ functional such that
\begin{equation}\label{cond5}
\begin{aligned}
&\lim_{\mu_{11},\mu_{22},\mu_{12}\rightarrow 0}\sup_{\|\vec{u}\|_{{\bf H}}\le M }| H_{\mu_{11},\mu_{22},\mu_{12}}(\vec{u})|=\lim_{\mu_{11},\mu_{22},\mu_{12}\rightarrow 0}\sup_{\|\vec{u}\|_{{\bf H}}\le M }\|H_{\mu_{11},\mu_{22},\mu_{12}}'(\vec{u})\|_{{\bf{H}}^{-1}}=0,\\
&H_{\mu_{11},\mu_{22},\mu_{12}}:{\bf{H}}\rightarrow \R \mbox{ and } H_{\mu_{11},\mu_{22},\mu_{12}}':{\bf{H}}\rightarrow {\bf{H}}^{-1} \mbox{ are compact}.
\end{aligned}
\end{equation}
Thus, by \eqref{cond1}, \eqref{cond22} and \eqref{cond3}-\eqref{cond5}, applying the result in \cite[Theorem 1]{1JS}, there exists a critical point  $\vec{u}_{2,\mu_{ij}}$ of $I_+$ such that $\vec{u}_{2,\mu_{ij}}$ converges to $(U,U)$ in $H^1(\R^3)\times H^1(\R^3)$ as $\mu_{ij}\rightarrow 0$. Then from this, \eqref{ibp} and  a strong maximum principle, we obtain that  $\vec{u}_{2,\mu_{ij}}$ is positive. Moreover, the convergence $\vec{u}_{2,\mu_{ij}}\rightarrow (U,U)$ and \eqref{cond12} imply that for small $\mu_{ij}>0$, $I_+(\vec{u}_{2,\mu_{ij}})>0.$
 \end{proof}

\subsection{Proof of Theorem \ref{th4}} 
In Proposition \ref{existwo}, we have shown the existence of a non-negative minimizer $\vec{u}_{1,\mu_{ij}}$  to \eqref{gme} with negative energy, and a  positive solution $\vec{u}_{2,\mu_{ij}}$ to \eqref{gme} with positive energy if   $\mu_{11}\mu_{22}-\mu_{12}^2>0$  and  $\mu_{ij}>0$ is small, where $i,j=1,2$.

We claim that a non-negative  minimizer $\vec{u}_{1,\mu_{ij}}$ of \eqref{gme} is positive.
Suppose to the contrary that a minimizer $\vec{u}_{1,\mu_{ij}}$ of \eqref{gme} is semi-trivial, that is, $\vec{u}_{1,\mu_{ij}}=(v_{1,\mu_{ij}},0).$ Then, by Proposition \ref{mmors}, we get a contradiction to the fact that $\vec{u}_{1,\mu_{ij}}$ is a minimizer of \eqref{gme}. Thus, we deduce that for small $\mu_{ij}>0$ with $\det(\mu_{ij}) > 0$, a minimizer $\vec{u}_1$ of \eqref{gme} is positive.

Finally, we further assume $\mu_{11}=\mu_{22}$. For small $\mu_{ij}>0$ with $\mu_{11}>\mu_{12}$, let  $\vec{u}_{1,\mu_{ij}}$ and $\vec{u}_{2,\mu_{ij}}$    be a positive minimizer of \eqref{gme} with negative energy and a positive solution of \eqref{gme} with positive energy, respectively. 
Then, by  Proposition \ref{rig1}, we deduce that $\vec{u}_{1,\mu_{ij}}$ and  $\vec{u}_{2,\mu_{ij}}$ have the form $\vec{u}_{1,\mu_{ij}}=(U_{1,\mu_{ij}},U_{1,\mu_{ij}})$  and  $\vec{u}_{2,\mu_{ij}}=(U_{2,\mu_{ij}},U_{2,\mu_{ij}})$, where $U_{1,\mu_{ij}}$ and $U_{2,\mu_{ij}}$ are   solutions of 
\begin{equation}\label{sie1}
 -\Delta u+ \lambda  u + ( {\mu_{11} - \mu_{12}}) \phi_{u} u= \frac{c_p}{2}|u|^{p-1} u. 
\end{equation}   We claim that for small $\mu_{ij}>0$ with $\mu_{11}>\mu_{12}$, $U_{1,\mu_{ij}}$ is a positive minimizer   of 
 \eqref{sie1} with negative energy and $U_{2,\mu_{ij}}$ is a positive solution  of 
 \eqref{sie1} with positive energy.
Since $I(u,u)=2\tilde{I}(u)$, where
$$
\tilde{I}(u)=\int_{\R^3}\frac12\Big(|\nabla u|^2+ \lambda u^2\Big) +\frac{\mu_{11}-\mu_{12}}{4}\phi_u u^2  -\frac{c_p}{2(p+1)}  |u|^{p+1}dx,
$$
we deduce that $U_{1,\mu_{ij}}$ is a minimizer of \eqref{sie1} with negative energy. 

On the other hand, by the construction in subsection \ref{p12}, we see that $U_{2,\mu_{ij}}\rightarrow U \not\equiv 0$ in $H^1(\R^3)$ as $\mu_{ij}\rightarrow 0$, where $i,j=1,2$, $U>0$ and $U$ satisfies \eqref{lseq}.  Then we see that as $\mu_{ij}\rightarrow 0$,
\[
\tilde{I}(U_{2,\mu_{ij}})\rightarrow \int_{\R^3}\frac12\Big(|\nabla U|^2+ \lambda U^2\Big)    -\frac{c_p}{2(p+1)}  |U|^{p+1}dx =\left(\frac12-\frac{1}{p+1}\right)\int_{\R^3}|\nabla U|^2+ \lambda U^2dx>0.
\]

\appendix
\section{Pohozaev's identity and Nonexistence result}\label{append1}
In this section, we prove Pohozaev's identity for \eqref{gme} and give the proof of Theorem \ref{th0}.

\begin{prop}\label{pz} 
Let $A,B,C$ and $D$ are real constants and $(u_1,u_2)\in {\bf H}$ be a weak solution of
\begin{equation}\label{pze}
\begin{cases} -A\Delta u_1+B\lambda u_1+C(\mu_{11}\phi_{u_1}-\mu_{12}\phi_{u_2})u_1=\frac{D}{2\pi}\int_0^{2\pi}|u_1+e^{i\theta}u_2|^{p-1}(u_1+e^{i\theta}u_2) d\theta
\\  -A\Delta u_2+B\lambda u_2+C(\mu_{22}\phi_{u_2}-\mu_{12}\phi_{u_1})u_2=\frac{D}{2\pi}\int_0^{2\pi} |u_2+e^{i\theta}u_1|^{p-1}(u_2+e^{i\theta}u_1)d\theta, \end{cases}
\end{equation}
where $\lambda, \mu_{ij} >0$ for  $i,j=1,2$, $-\Delta \phi_u=u^2$ and   $1<p<5$. Then it holds
\begin{equation}\label{pz-identity}
\begin{aligned}
0 &= \int_{\R^3} \frac{A}{2}\Big(|\nabla u_1|^2+|\nabla u_2|^2\Big)+\frac{3B}{2}\lambda( u_1^2+ u_2^2)+\frac{5C}{4} \Big(\mu_{11}|\nabla\phi_{u_1}|^2+\mu_{22}|\nabla\phi_{u_2}|^2 -2 \mu_{12}  \nabla \phi_{u_1} \cdot \nabla \phi_{u_2} \Big)  \\
&\qquad -\frac{3D}{p+1}\bigg[\frac{1}{2\pi}\int_0^{2\pi} \Big(u_1^2+2u_1 u_2 \cos \theta+u_2^2\Big)^\frac{p+1}{2}d\theta\bigg] dx.
\end{aligned}
\end{equation}
\end{prop}
\begin{proof}
Let
$$(u,v)\in {\bf H}\cap \Big(H^2_{loc}(\R^3)\times H^2_{loc}(\R^3)\Big) \mbox{ and }(\phi_u,\phi_v)\in  \Big(D^{1,2}(\R^3)\times D^{1,2}(\R^3)\Big)\cap \Big(H^2_{loc}(\R^3)\times H^2_{loc}(\R^3)\Big),$$
where $-\Delta \phi_u=u^2$ and $-\Delta \phi_v=v^2$.
By \eqref{uq} and integration by parts, we have
\begin{equation}\label{f0}
\begin{aligned}
&\int_{B(0,R)}\int_0^{2\pi}  |u+e^{i\theta}v|^{p-1}(u+e^{i\theta}v)(x\cdot \nabla u)+ |v+e^{i\theta}u|^{p-1}(v+e^{i\theta}u)(x\cdot \nabla v) d\theta dx\\
&=\int_0^{2\pi}\int_{B(0,R)}\Big((u+v \cos \theta )(x\cdot \nabla u)+(v+u \cos \theta )(x\cdot \nabla v)\Big)\Big(u^2+2u v \cos \theta+v^2\Big)^\frac{p-1}{2}  dxd\theta\\
&=\frac{1}{p+1}\int_0^{2\pi}\int_{B(0,R)}x\cdot \nabla \Big(u^2+2u v \cos \theta+v^2\Big)^\frac{p+1}{2}dxd\theta\\
&=\int_0^{2\pi}\bigg[-\frac{3}{p+1}\int_{B(0,R)} \Big(u^2+2u v \cos \theta+v^2\Big)^\frac{p+1}{2}dx\\
&\quad +\frac{R}{p+1}\int_{\partial B(0,R)}(u^2+2u v \cos \theta+v^2 )^\frac{p+1}{2}d\sigma \bigg]d\theta.
\end{aligned}
\end{equation}
On the other hand, by integration by parts,
\begin{equation*}\label{a2}
\begin{aligned}
&\int_{B(0,R)}(-\Delta \phi_u)(x\cdot \nabla \phi_v)dx = -\int_{B(0,R)}(\partial_{jj}\phi_u) (x_i \partial_i \phi_v) dx\\
&=\int_{B(0,R)}(\partial_i \phi_u) ( \partial_i \phi_v) +x_i(\partial_j \phi_u )(\partial_{ij}\phi_v)dx-\int_{\partial B(0,R)}(\partial_j \phi_u) (\partial_i \phi_v) \frac{x_i x_j}{|x|}d\sigma\\
&=\int_{B(0,R)}(\partial_i \phi_u) ( \partial_i \phi_v) -3 (\partial_j \phi_u )(\partial_j \phi_v)-x_i(\partial_{ij}\phi_u)(\partial_j\phi_v)dx\\
&\quad -\int_{\partial B(0,R)}(\partial_j \phi_u) (\partial_i \phi_v) \frac{x_i x_j}{|x|}d\sigma+\int_{\partial B(0,R)}|x|(\partial_j\phi_u) (\partial_j\phi_v) d\sigma.
\end{aligned}
\end{equation*}
Here, we adopt the summation convention on repeated indices.
From this, we see that  the term $\int_{B(0,R)}u^2(x\cdot \nabla \phi_v)dx$ can be expressed in the following two ways:
\begin{equation}\label{pze2}
\begin{aligned}
\int_{B(0,R)}u^2(x\cdot \nabla \phi_v)dx&=\int_{B(0,R)} \nabla \phi_u \cdot \nabla  \phi_v  +x_i(\partial_j \phi_u )(\partial_{ij}\phi_v)dx\\
&\quad -\frac{1}{R}\int_{\partial B(0,R)}(x\cdot \nabla \phi_u)(x\cdot \nabla \phi_v)d\sigma
\end{aligned}
\end{equation}
and
\begin{equation}\label{pze3}
\begin{aligned}
 \int_{B(0,R)}u^2(x\cdot \nabla \phi_v)dx&=\int_{B(0,R)} -2\nabla \phi_u \cdot \nabla \phi_v   -x_i(\partial_{ij}\phi_u)(\partial_j\phi_v)dx \\
&\quad +\int_{\partial B(0,R)}R\nabla \phi_u\cdot \nabla \phi_v-\frac{1}{R}(x\cdot \nabla \phi_u)(x\cdot \nabla \phi_v)d\sigma.
\end{aligned}
\end{equation}
Then, by \eqref{pze2} and \cite[Lemma 3.1 (3.13)]{DM},
\begin{equation}\label{pze4}
\begin{aligned}
&\int_{B(0,R)}\phi_v u(x\cdot \nabla u)dx\\
&=-\int_{B(0,R)}\frac12 u^2(x\cdot \nabla \phi_v)+\frac32 \phi_v u^2dx+\frac{R}{2}\int_{\partial B(0,R)}\phi_v u^2d\sigma\\
&=-\int_{B(0,R)}\frac12 u^2(x\cdot \nabla \phi_v)+\frac32 \phi_v (-\Delta \phi_u)dx+\frac{R}{2}\int_{\partial B(0,R)}\phi_v u^2d\sigma\\
&=-\int_{B(0,R)}\frac12 u^2(x\cdot \nabla \phi_v)+\frac32 \nabla  \phi_u\cdot\nabla \phi_v  dx+\int_{\partial B(0,R)}\frac{R}{2}\phi_v u^2+\frac{3}{2R}\phi_v (\nabla\phi_u\cdot x)d\sigma\\
&= \int_{B(0,R)}-2\nabla \phi_u \cdot \nabla \phi_v -\frac12 x_i(\partial_j \phi_u )(\partial_{ij}\phi_v)dx\\
&\quad +\int_{\partial B(0,R)}\frac{R}{2}\phi_v u^2+\frac{3}{2R} \phi_v (\nabla\phi_u\cdot x)+\frac{1}{2R}(x\cdot \nabla \phi_u)(x\cdot \nabla \phi_v)d\sigma
\end{aligned}
\end{equation}
and similarily, by \eqref{pze3} and \cite[Lemma 3.1 (3.13)]{DM},
\begin{equation}\label{pze5} 
\begin{aligned}
&\int_{B(0,R)}\phi_u v(x\cdot \nabla v)dx\\
& =-\int_{B(0,R)}\frac12 v^2(x\cdot \nabla \phi_u)+\frac32 \phi_u v^2dx+\frac{R}{2}\int_{\partial B(0,R)}\phi_u v^2d\sigma\\ 
&=-\int_{B(0,R)}\frac12 v^2(x\cdot \nabla \phi_u)+\frac32 \nabla  \phi_u\cdot\nabla \phi_v  dx+\int_{\partial B(0,R)}\frac{R}{2}\phi_u v^2+\frac{3}{2R} \phi_u (\nabla \phi_v\cdot x)d\sigma\\
&=\int_{B(0,R)} -\frac12 \nabla \phi_u \cdot \nabla \phi_v  +\frac12 x_i(\partial_j\phi_u)(\partial_{ij}\phi_v)dx\\
&\quad +\int_{\partial B(0,R)}\frac{R}{2}\phi_u v^2+\frac{3}{2R} \phi_u (\nabla \phi_v\cdot x)-\frac{R}{2}\nabla \phi_u\cdot \nabla \phi_v+\frac{1}{2R}(x\cdot \nabla \phi_u)(x\cdot \nabla \phi_v)d\sigma.
\end{aligned}
\end{equation}
Thus, by \eqref{pze4} and \eqref{pze5}, we see that
\begin{equation}\label{f1}
\begin{aligned}
&\int_{B(0,R)}\phi_v u(x\cdot \nabla u)+\phi_u v(x\cdot \nabla v)dx\\
&= \int_{B(0,R)}-\frac52\nabla \phi_u \cdot \nabla \phi_v dx   +\int_{\partial B(0,R)}\frac{R}{2}(\phi_v u^2+\phi_u v^2)+\frac{3}{2R} \Big(\phi_v (\nabla\phi_u\cdot x)+\phi_u (\nabla\phi_v\cdot x)\Big)\\
&\quad-\frac{R}{2}\nabla \phi_u\cdot \nabla \phi_v+\frac{1}{ R}(x\cdot \nabla \phi_u)(x\cdot \nabla \phi_v)d\sigma.
\end{aligned}
\end{equation}
Moreover, by \cite[Lemma 3.1]{DM},
\begin{equation}\label{f2}
\begin{aligned}
&\int_{B(0,R)}( -A\Delta u+B\lambda  u)(x\cdot \nabla  u)dx\\
&=-\int_{B(0,R)}\frac{A}{2}|\nabla u|^2+\frac{3B\lambda }{2}u^2dx+\int_{\partial B(0,R)}-\frac{A}{R}|x\cdot \nabla u|^2+\frac{RA}{2}|\nabla u|^2+\frac{RB\lambda}{2}u^2d\sigma
\end{aligned}
\end{equation}
and
\begin{equation}\label{f3}
\begin{aligned}
&\int_{B(0,R)}\phi_u u(x\cdot \nabla u)dx\\
&=-\int_{B(0,R)}\frac12 u^2(x\cdot \nabla \phi_u)+\frac32 \phi_u u^2dx+\frac{R}{2}\int_{\partial B(0,R)}\phi_u u^2d\sigma\\
&=-\int_{B(0,R)}\frac12 (-\Delta \phi_u)(x\cdot \nabla \phi_u)+\frac32 \phi_u u^2dx+\frac{R}{2}\int_{\partial B(0,R)}\phi_u u^2d\sigma\\
&=\int_{B(0,R)}\frac14|\nabla \phi_u|^2-\frac32 \phi_u u^2dx+\int_{\partial B(0,R)} \frac{1}{2R}|x\cdot \nabla \phi_u|^2-\frac{R}{4}|\nabla \phi_u|^2+\frac{R}{2}\phi_u u^2d\sigma.
\end{aligned}
\end{equation}

Let $(u_1,u_2)\in {\bf H}$ be a weak solution of \eqref{pze}. Then, by Newton's Theorem, the Strauss inequality and \cite[Theorem 8.17]{GT}, $(\phi_{u_1}, \phi_{u_2})\in  L^\infty_{loc}(\R^3)\times L^\infty_{loc}(\R^3)$, $(u_1,u_2)\in L^\infty_{loc}(\R^3)\times L^\infty_{loc}(\R^3)$, and thus standard regularity theory \cite[Theorem 9.19]{GT} implies $u_1, u_2, \phi_{u_1}, \phi_{u_2}\in W^{2,q}_{loc}(\R^3)$ for any $q\in [1,\infty)$. Moreover, by a comparision principle, we have
$$
u_i(x)\le \tilde{C}_1\exp(-\tilde{C}_2|x|) \mbox{ for } x\in \R^3,
$$
where $i=1,2$ and $\tilde{C}_1, \tilde{C}_2>$ are constants.
Since $(u_1,u_2)\in {\bf H}$ is a weak solution of \eqref{pze}, we have
\begin{equation}\label{f4}
\begin{aligned}
0&=\int_{B(0,R)}    
(-A\Delta u_1+ B\lambda u_1)(x\cdot \nabla u_1)+ (-A\Delta u_2+B\lambda u_2)(x\cdot \nabla u_2)\\
&\quad + C \bigg(\mu_{11}\phi_{u_1}u_1(x\cdot \nabla u_1)+\mu_{22}\phi_{u_2}u_2(x\cdot \nabla u_2)-\mu_{12}\Big(\phi_{u_2} u_1(x\cdot \nabla u_1)+\phi_{u_1} u_2(x\cdot \nabla u_2)\Big)\bigg)  \\
&\quad -  \frac{D}{2\pi}\int_0^{2\pi}  |u_1+e^{i\theta}u_2|^{p-1}(u_1+e^{i\theta}u_2)  (x\cdot \nabla u_1)+|u_2+e^{i\theta}u_1|^{p-1}(u_2+e^{i\theta}u_1)(x\cdot \nabla u_2) d\theta\  dx. 
\end{aligned}
\end{equation}
Then, by \eqref{f0} and \eqref{f1}-\eqref{f4}, we deduce that 
\begin{equation}\label{f}
\begin{aligned}
 &\int_{B(0,R)} \frac{A}{2}\Big(|\nabla u_1|^2+|\nabla u_2|^2\Big)+\frac{3B}{2}\lambda( u_1^2+ u_2^2)\\
&\quad +C\Big(-\frac{\mu_{11}}{4}|\nabla\phi_{u_1}|^2+\frac{3\mu_{11}}{2}\phi_{u_1}u_1^2-\frac{\mu_{22}}{4}|\nabla\phi_{u_2}|^2+\frac{3\mu_{22}}{2}\phi_{u_2}u_2^2  -\frac52 \mu_{12}  \nabla \phi_{u_1} \cdot \nabla \phi_{u_2} \Big)  \\
&\quad -\frac{3D}{p+1}\bigg[\frac{1}{2\pi}\int_0^{2\pi} \Big(u_1^2+2u_1 u_2 \cos \theta+u_2^2\Big)^\frac{p+1}{2}d\theta\bigg] dx\\
&= \int_{\partial B(0,R)}-\frac{A}{R}|x\cdot \nabla u_1|^2+\frac{RA}{2}|\nabla u_1|^2+\frac{RB\lambda}{2}u_1^2-\frac{A}{R}|x\cdot \nabla u_2|^2+\frac{RA}{2}|\nabla u_2|^2+\frac{RB\lambda}{2}u_2^2\\
& +C\bigg[\frac{\mu_{11}}{2R}|x\cdot \nabla \phi_{u_1}|^2-\frac{R\mu_{11}}{4}|\nabla \phi_{u_1}|^2+\frac{R\mu_{11}}{2}\phi_{u_1} u_1^2+\frac{\mu_{22}}{2R}|x\cdot \nabla \phi_{u_2}|^2-\frac{R\mu_{22}}{4}|\nabla \phi_{u_2}|^2+\frac{R\mu_{22}}{2}\phi_{u_2} u_2^2\\
&\quad -\frac{R\mu_{12}}{2}(\phi_{u_2} u_1^2+\phi_{u_1} u_2^2)-\frac{3\mu_{12}}{2R} \Big(\phi_{u_2} (\nabla\phi_{u_1}\cdot x)+\phi_{u_1} (\nabla\phi_{u_2}\cdot x)\Big) +\frac{R\mu_{12}}{2}\nabla \phi_{u_1}\cdot \nabla \phi_{u_2}\\
&\quad -\frac{\mu_{12}}{ R}(x\cdot \nabla \phi_{u_1})(x\cdot \nabla \phi_{u_2})\bigg] -\frac{RD}{2\pi(p+1)} \int_0^{2\pi}(u_1^2+2u_1 u_2 \cos \theta+u_2^2 )^\frac{p+1}{2}d\theta  d\sigma.
\end{aligned}
\end{equation}
We will prove that the right hand side in \eqref{f} converges to zero for a suitable sequence $R_n\rightarrow \infty$. Assuming this, we have
\begin{align*}
0&=\int_{\R^3} \frac{A}{2}\Big(|\nabla u_1|^2+|\nabla u_2|^2\Big)+\frac{3B}{2}\lambda( u_1^2+ u_2^2)\\
&\quad +C\Big(-\frac{\mu_{11}}{4}|\nabla\phi_{u_1}|^2+\frac{3\mu_{11}}{2}\phi_{u_1}u_1^2-\frac{\mu_{22}}{4}|\nabla\phi_{u_2}|^2+\frac{3\mu_{22}}{2}\phi_{u_2}u_2^2  -\frac52 \mu_{12}  \nabla \phi_{u_1} \cdot \nabla \phi_{u_2} \Big)  \\
&\quad -\frac{3D}{p+1}\bigg[\frac{1}{2\pi}\int_0^{2\pi} \Big(u_1^2+2u_1 u_2 \cos \theta+u_2^2\Big)^\frac{p+1}{2}d\theta\bigg] dx\\
 &=\int_{\R^3} \frac{A}{2}\Big(|\nabla u_1|^2+|\nabla u_2|^2\Big)+\frac{3B}{2}\lambda( u_1^2+ u_2^2)+\frac{5C}{4} \Big(\mu_{11}|\nabla\phi_{u_1}|^2+\mu_{22}|\nabla\phi_{u_2}|^2 -2 \mu_{12}  \nabla \phi_{u_1} \cdot \nabla \phi_{u_2} \Big)  \\
&\qquad -\frac{3D}{p+1}\bigg[\frac{1}{2\pi}\int_0^{2\pi} \Big(u_1^2+2u_1 u_2 \cos \theta+u_2^2\Big)^\frac{p+1}{2}d\theta\bigg] dx.
\end{align*}

Following the arguments in \cite[Proposition 1]{BL}, we can prove that the right hand side in \eqref{f} converges to zero for a suitable sequence $R_n\rightarrow \infty$. Indeed, note that by \eqref{phexp},
$$
\phi_{u_i}(x)\le \frac{\|u_i\|_{H^1}^2}{|x|} \ \   \mbox{ and }\ \   |\nabla \phi_{u_i}(x)| \le\frac{\|u_i\|_{H^1}^2}{|x|^2}. 
$$
From this, we see that
\begin{equation}\label{pzee1}
 \Big| \int_{\partial B(0,R)}\frac{1}{ R} \phi_{u_i} (\nabla \phi_{u_j}\cdot x)d\sigma\Big| \le CR^{-1},
\end{equation}
where $C$ is a constant independent of $R$. Moreover, by \cite[Proposition 1]{BL}, if $G\in L^1(\R^3)$, then there exists $R_n\rightarrow \infty$ such that
\begin{equation}\label{pzee2}
R_n\int_{\partial B(0,R_n)}|G|d\sigma\rightarrow 0 \mbox{ as } n\rightarrow \infty.
\end{equation}
Then, by \eqref{pzee1}, \eqref{pzee2}, Sobolev embedding theorem and the fact  
\begin{equation*} 
\int_{\R^3}\phi_v u^2dx=\int_{\R^3}\nabla \phi_u\cdot \nabla \phi_v dx\le \int_{\R^3}|\nabla \phi_u|| \nabla \phi_v|dx\le\|\phi_u\|_{D^{1,2}} \|\phi_v\|_{D^{1,2}},
\end{equation*}
we can prove the claim.
\end{proof}

\noindent {\bf Proof of Theorem \ref{th0}}
As a direction application of Pohozaev's identity \eqref{pz-identity}, we prove Theorem \ref{th0}, the nonexistence of nontrivial solutions to \eqref{gme} when $\mu_{11}\mu_{22}-\mu_{12}^2\ge0$ and $p\in (0,1]\cup[5,\infty)$.
Assume that $\vec{u}=(u_1,u_2)\in H^1(\R^3)\times H^1(\R^3)$ satisfies \eqref{gme}. Then, by Proposition \ref{pz} and the fact that $I^\prime(\vec{u})\vec{u}=0$, we see that
\begin{equation}\label{hc1}
\begin{aligned}
0=&\int_{\R^3}|\nabla u_1|^2+|\nabla u_2|^2+\lambda u_1^2+\lambda u_2^2dx+\int_{\R^3} \mu_{11}u_1^2 \phi_{u_1} + \mu_{22}u_2^2 \phi_{u_2}-2\mu_{12} u_1^2 \phi_{u_2} dx\\
&\qquad -  \frac{1}{2\pi}  \int_{\R^3}\int_0^{2\pi}\Big(u_1^2+2u_1 u_2 \cos \theta+u_2^2\Big)^\frac{p+1}{2} d\theta dx
\end{aligned}
\end{equation}
and
\begin{equation}\label{hc2}
\begin{aligned}
0 &= \int_{\R^3} |\nabla u_1|^2+|\nabla u_2|^2 +3\lambda( u_1^2+ u_2^2)dx+\frac{5 }{2} \int_{\R^3}\mu_{11}u_1^2 \phi_{u_1} + \mu_{22}u_2^2 \phi_{u_2}-2\mu_{12} u_1^2 \phi_{u_2} dx\\
&\qquad -\frac{6 }{p+1}\bigg[\frac{1}{2\pi}\int_{\R^3}\int_0^{2\pi} \Big(u_1^2+2u_1 u_2 \cos \theta+u_2^2\Big)^\frac{p+1}{2}d\theta dx\bigg].
\end{aligned}
\end{equation}
Subtracting the first equation \eqref{hc1} from the second one \eqref{hc2}, we have
\begin{align*}
0=&\int_{\R^3} 2\lambda( u_1^2+ u_2^2)dx+\frac{3 }{2} \int_{\R^3}\mu_{11}u_1^2 \phi_{u_1} + \mu_{22}u_2^2 \phi_{u_2}-2\mu_{12} u_1^2 \phi_{u_2} dx\\
&\qquad + \Big( 1-\frac{6 }{p+1}\Big)\bigg[\frac{1}{2\pi}\int_{\R^3}\int_0^{2\pi} \Big(u_1^2+2u_1 u_2 \cos \theta+u_2^2\Big)^\frac{p+1}{2}d\theta dx\bigg].
\end{align*}
From this, \eqref{mupos} and the facts that $\mu_{11}>0, \mu_{11}\mu_{22}-\mu_{12}^2\ge 0$, we see that $(u_1, u_2)=(0,0)$ if $p\ge 5$.

On the other hand, we multiply the second equation \eqref{hc2} by $\frac13$ and subtract the calculated  one from the first equation \eqref{hc1}, we have
\begin{align*}
0=&\frac23\int_{\R^3}|\nabla u_1|^2+|\nabla u_2|^2 dx+\frac16\int_{\R^3} \mu_{11}u_1^2 \phi_{u_1} + \mu_{22}u_2^2 \phi_{u_2}-2\mu_{12} u_1^2 \phi_{u_2} dx\\
&\qquad +\Big(\frac{2}{p+1}-1\Big)   \frac{1}{2\pi}  \int_{\R^3}\int_0^{2\pi}\Big(u_1^2+2u_1 u_2 \cos \theta+u_2^2\Big)^\frac{p+1}{2} d\theta dx.
\end{align*}
From this, \eqref{mupos} and the facts that $\mu_{11}>0, \mu_{11}\mu_{22}-\mu_{12}^2\ge 0$, we see that $(u_1, u_2)=(0,0)$ if $p\le 1$.  $\Box$

\section{Regularity of the action functional $I(\vec{u})$}\label{regularity}
In this section, we study regularities for the action function $I$ for \eqref{gme}. We first show \eqref{gme} is the Euler-Lagrange equation of $I$.

\begin{prop}
The functional $I$ is well-defined and continuously differentiable on ${\bf H}$.
In addition, any critical point of $I$ is a weak solution of \eqref{gme}.
\end{prop}
\begin{proof}
Standard computation shows that $I$ is $C^1$ on $\mathbf{H}$ and satisfies
\begin{equation}\label{oned}
\begin{aligned}
I^\prime(\vec{u})\vec{\psi}&= \int_{\R^3} \nabla u_1\cdot \nabla \psi_1 + \nabla u_2\cdot \nabla \psi_2+ \lambda u_1 \psi_1+\lambda u_2\psi_2dx\\
&\quad + \int_{\R^3}\mu_{11}   u_1 \phi_{u_1} \psi_1 + \mu_{22}    u_2  \phi_{u_2} \psi_2 - \mu_{12} u_1  \phi_{u_2} \psi_1 -\mu_{12}  u_2 \phi_{u_1} \psi_2 dx\\
& \quad -  \frac{1}{2\pi}\int_{\R^3}\int_0^{2\pi} \Big((u_1+u_2 \cos \theta ) \psi_1+ (u_2+u_1 \cos \theta )\psi_2\Big)\Big(u_1^2+2u_1 u_2 \cos \theta+u_2^2\Big)^\frac{p-1}{2} d\theta dx\\
&=  \int_{\R^3} \nabla u_1\cdot \nabla \psi_1 + \nabla u_2\cdot \nabla \psi_2+ \lambda u_1 \psi_1+\lambda u_2\psi_2dx\\
&\quad + \int_{\R^3}\mu_{11}   u_1 \phi_{u_1} \psi_1 + \mu_{22}    u_2  \phi_{u_2} \psi_2 - \mu_{12} u_1  \phi_{u_2} \psi_1 -\mu_{12}  u_2 \phi_{u_1} \psi_2 dx\\
& \quad -  \frac{1}{2\pi}\int_{\R^3}\int_0^{2\pi}  |u_1+e^{i\theta}u_2|^{p-1}(u_1+e^{i\theta}u_2)\psi_1+ |u_2+e^{i\theta}u_1|^{p-1}(u_2+e^{i\theta}u_1)\psi_2    d\theta dx
\end{aligned}
\end{equation}
which implies that any critical point of $I$ is a weak solution of \eqref{gme}.
Here we used the fact that for $a,b\in \R$,
\begin{equation}\label{uq}
\begin{aligned}
\int_0^{2\pi}(a+b e^{i\theta})|a+b e^{i\theta}|^{p-1}d\theta&=\int_0^{2\pi}(a+b \cos \theta +ib\sin\theta)\Big((a+b \cos \theta)^2+(b\sin\theta)^2\Big)^\frac{p-1}{2}d\theta\\
&=\int_0^{2\pi}(a+b \cos \theta )\Big(a^2+2ab \cos \theta+b^2\Big)^\frac{p-1}{2}d\theta.
\end{aligned}
\end{equation}
\end{proof}

We note that a nontrivial critical point of $I$ needs not to be nonnegative. Thus, in order to find a nonnegative solution of \eqref{gme}, we consider a truncated functional $I_+:{\bf H}\rightarrow \R$ given by
\begin{equation}\label{funpl}
\begin{aligned}
I_+(\vec{u})=&\frac12\int_{\R^3}|\nabla u_1|^2+|\nabla u_2|^2+\lambda u_1^2+\lambda u_2^2dx\\
&\quad +\frac14\int_{\R^3} \mu_{11}(u_1)_+^2 \phi_{(u_1)_+} + \mu_{22}(u_2)_+^2 \phi_{(u_2)_+}-2\mu_{12} (u_1)_+^2 \phi_{(u_2)_+} dx\\
&\quad - \frac{1}{p+1} \frac{1}{2\pi}\int_{\R^3}\int_0^{2\pi} \Big((u_1)_+^2+2(u_1)_+ (u_2)_+ \cos \theta+(u_2)_+^2\Big)^\frac{p+1}{2}d\theta dx,
\end{aligned}
\end{equation} 
where  $u_+=\max\{u,0\}$.
\begin{prop}
The functional $I_+$ is continuously differentiable on $\mathbf{H}$.
In addition, any critical point of $I_+$ is a nonnegative weak solution of \eqref{gme}.
\end{prop}
\begin{proof}
Let us denote
$$
A_+(\vec{u})=\frac{1}{p+1} \int_{\R^3}\int_0^{2\pi} \Big((u_1)_+^2+2(u_1)_+ (u_2)_+ \cos \theta+(u_2)_+^2\Big)^\frac{p+1}{2}d\theta dx.
$$
 To prove that $A_+$ is $C^1$,  with derivative  given by
\begin{equation}\label{ibppa}
\begin{aligned}
A^\prime_+(\vec{u})[\psi_1,\psi_2]&=\int_{\R^3} \int_0^{2\pi} \Big((u_1)_+^2+2(u_1)_+ (u_2)_+ \cos \theta+(u_2)_+^2\Big)^\frac{p-1}{2}\\
&\qquad \times \left(\Big((u_1)_++1_{\{u_1>0\}}(u_2)_+\cos\theta\Big)\psi_1+\Big((u_2)_++1_{\{u_2>0\}}(u_1)_+\cos\theta\Big)\psi_2\right) d\theta dx,
\end{aligned}
\end{equation}
we will show that the map $\vec{u}\mapsto A^\prime_+(\vec{u})$ is continuous.
We define 
$$\alpha(t)=\int_0^{2\pi}(1+t \cos \theta )(1+2t \cos \theta+t^2)^\frac{p-1}{2}d\theta.
$$
Then since for large $t>0$,
\begin{align*}
\alpha(t)/t^{p-1}&=\int_0^{2\pi}(1+t \cos \theta )\Big(1+2t^{-1} \cos \theta+t^{-2}\Big)^\frac{p-1}{2}d\theta\\
&=\int_0^{2\pi}(1+t \cos \theta )\Big(1+\frac{p-1}{2}(2t^{-1} \cos \theta+t^{-2})+O(t^{-2})\Big) d\theta\\
&=\int_0^{2\pi}1+(p-1)(\cos\theta)^2 d\theta+O(t^{-1})=\pi(p+1)+O(t^{-1}),
\end{align*}
we see that    $\lim_{t\rightarrow 0}\alpha(t)=2\pi$ and $\lim_{t\rightarrow \infty}\alpha(t)/t^{p-1}=\pi(p+1)$.
Since $\alpha$ is continuous, we see that 
\begin{equation*}
\alpha(t)\le C(1+t^{p-1}) \mbox{ for all } t\ge0.
\end{equation*} From this, we get
\begin{equation}\label{ibp}
\begin{aligned}
& \int_0^{2\pi} \Big((u_1)_+^2+2(u_1)_+ (u_2)_+ \cos \theta+(u_2)_+^2\Big)^\frac{p-1}{2}\left((u_1)_++1_{\{u_1>0\}}(u_2)_+\cos\theta\right) d\theta   \\
&= (u_1)_+^p \int_0^{2\pi} \Big((1+2(u_1)_+^{-1} (u_2)_+ \cos \theta+(u_1)_+^{-2}(u_2)_+^2\Big)^\frac{p-1}{2}\left(1+(u_1)_+^{-1}(u_2)_+\cos\theta\right) d\theta   \\
&\le C  (u_1)_+\left( (u_1)_+^{p-1}+ (u_2)_+^{p-1}\right).
\end{aligned}
\end{equation}
From this, we see that
\begin{align*}
&\bigg\|\int_0^{2\pi} \Big((u_1)_+^2+2(u_1)_+ (u_2)_+ \cos \theta+(u_2)_+^2\Big)^\frac{p-1}{2}\left((u_1)_++1_{\{u_1>0\}}(u_2)_+\cos\theta\right) \bigg\|_{L^\frac{p+1}{p}}\\
&\le C\left\|(u_1)_+\left( (u_1)_+^{p-1}+ (u_2)_+^{p-1}\right)\right\|_{L^\frac{p+1}{p}}\le C(\|u_1\|_{L^{p+1}}^p+\|u_2\|_{L^{p+1}}^p),
\end{align*} 
and hence if $(u_1,u_2)\rightarrow (v_1,v_2)$ in ${\bf H}$, then for $\psi_1\in H^1(\R^3)$,
\begin{align*}
&\int_{\R^3}\int_0^{2\pi} \Big((u_1)_+^2+2(u_1)_+ (u_2)_+ \cos \theta+(u_2)_+^2\Big)^\frac{p-1}{2}\left((u_1)_++1_{\{u_1>0\}}(u_2)_+\cos\theta\right)\\
&\quad -\Big((v_1)_+^2+2(v_1)_+ (v_2)_+ \cos \theta+(v_2)_+^2\Big)^\frac{p-1}{2}\left((v_1)_++1_{\{v_1>0\}}(v_2)_+\cos\theta\right)  d\theta  \psi_1 dx\\
&\le \|\psi_1\|_{L^{p+1} }\bigg\|\int_0^{2\pi} \Big((u_1)_+^2+2(u_1)_+ (u_2)_+ \cos \theta+(u_2)_+^2\Big)^\frac{p-1}{2}\left((u_1)_++1_{\{u_1>0\}}(u_2)_+\cos\theta\right)\\
&\quad -\Big((v_1)_+^2+2(v_1)_+ (v_2)_+ \cos \theta+(v_2)_+^2\Big)^\frac{p-1}{2}\left((v_1)_++1_{\{v_1>0\}}(v_2)_+\cos\theta\right)  d\theta  \bigg\|_{L^\frac{p+1}{p}}\rightarrow 0.
\end{align*} 
This implies that the map $\vec{u}\mapsto A^\prime_+(\vec{u})$ is continuous.

Then we see that  $I_+$ is $C^1$ and
\begin{align*}
&I^\prime_+(\vec{u})[\psi_1,\psi_2]\\
&=\int_{\R^3} \nabla u_1\cdot \nabla \psi_1 + \nabla u_2\cdot \nabla \psi_2+\lambda u_1 \psi_1+\lambda u_2\psi_2dx\\
&\quad + \int_{\R^3}\mu_{11}   (u_1)_+ \phi_{(u_1)_+} \psi_1 + \mu_{22}    (u_2)_+  \phi_{(u_2)_+} \psi_2 - \mu_{12} (u_1)_+  \phi_{(u_2)_+} \psi_1 -\mu_{12}  (u_2)_+ \phi_{(u_1)_+} \psi_2 dx\\
&\quad - \frac{1}{2\pi}\int_{\R^3} \int_0^{2\pi} \Big((u_1)_+^2+2(u_1)_+ (u_2)_+ \cos \theta+(u_2)_+^2\Big)^\frac{p-1}{2}\\
&\qquad \times \left(\Big((u_1)_++1_{\{u_1>0\}}(u_2)_+\cos\theta\Big)\psi_1+\Big((u_2)_++1_{\{u_2>0\}}(u_1)_+\cos\theta\Big)\psi_2\right) d\theta dx.
\end{align*} 
 Moreover, from  a maximum principle, we see that any critical point of $I_+$ is a non-negative solution of \eqref{gme}.
\end{proof}

We next show that $I$ is actually $C^2$. This regularity is required to compute the Morse indices of critical points of $I$.
\begin{prop}
The functional $I$ is twice continuously differentiable on $\mathbf{H}$.
\end{prop}
\begin{proof}
Let  $\vec{\psi}=(\psi_1,\psi_2)$,
$$
T(\vec{u})=\frac14\int_{\R^3} \mu_{11}u_1^2 \phi_{u_1} + \mu_{22}u_2^2 \phi_{u_2}-2\mu_{12} u_1^2 \phi_{u_2} dx.
$$
and
$$A(\vec{u})=\frac{1}{p+1}  \int_{\R^3}\int_0^{2\pi} \Big(u_1^2+2u_1 u_2 \cos \theta+u_2^2\Big)^\frac{p+1}{2}d\theta dx.
$$ 
Note that  $T$ is $C^2$ functional, and hence  we have
$$
T^\prime(\vec{u})\vec{\psi}=\int_{\R^3}\mu_{11}   u_1 \phi_{u_1} \psi_1 + \mu_{22}    u_2  \phi_{u_2} \psi_2 - \mu_{12} u_1  \phi_{u_2} \psi_1 -\mu_{12}  u_2 \phi_{u_1} \psi_2 dx
$$
and
\begin{equation}\label{esee2}
\begin{aligned}
T^{\prime\prime}(\vec{u})[\vec{\psi},\vec{\psi}]&=\int_{\R^3}\mu_{11}   \phi_{u_1} \psi_1^2+\mu_{11}   u_1 (\phi_{u_1}^\prime \psi_1) \psi_1- \mu_{12}   \phi_{u_2} \psi_1^2 -\mu_{12}  u_2 (\phi_{u_1}^\prime \psi_1) \psi_2 \\
&\quad  + \mu_{22}     \phi_{u_2} \psi_2^2 + \mu_{22}    u_2  (\phi_{u_2}^\prime \psi_2) \psi_2  -\mu_{12} u_1 (\phi_{u_2}^\prime \psi_2)\psi_1-\mu_{12}\phi_{u_1}\psi_2^2dx\\
&=  \int_{\R^3}\mu_{11}   \phi_{u_1} \psi_1^2+\mu_{11}   u_1 (\phi_{u_1}^\prime \psi_1) \psi_1- \mu_{12}   \phi_{u_2} \psi_1^2 -2\mu_{12}  u_2 (\phi_{u_1}^\prime \psi_1) \psi_2 \\
&\quad  + \mu_{22}     \phi_{u_2} \psi_2^2 + \mu_{22}    u_2  (\phi_{u_2}^\prime \psi_2) \psi_2   -\mu_{12}\phi_{u_1}\psi_2^2dx.
\end{aligned}
\end{equation}
Moreover, note that since
$$
u_1^2+2u_1 u_2 \cos \theta+u_2^2=(u_1+u_2\cos \theta )^2+u_2^2\sin^2\theta=(u_2+u_1\cos\theta)^2+u_1^2\sin^2\theta,
$$
we have  
\begin{equation}\label{nes}
(u_1+u_2 \cos \theta )^2\Big(u_1^2+2u_1 u_2 \cos \theta+u_2^2\Big)^{-1}\le 1
\end{equation}
and 
$$
(u_2+u_1 \cos \theta )(u_1+u_2\cos\theta)\Big(u_2^2+2u_2 u_1 \cos \theta+u_1^2\Big)^{-1}\le 1,
$$
and hence we get
$$
\left\|\int_0^{2\pi}(u_2+u_1 \cos \theta )(u_1+u_2\cos\theta)\Big(u_1^2+2u_1 u_2 \cos \theta+u_2^2\Big)^\frac{p-3}{2}d\theta \right\|_{L^\frac{p+1}{p-1}}\le C(\|u_1\|_{L^{p+1}}^{p-1}+\|u_2\|_{L^{p+1}}^{p-1})
$$
From this, we see that if $(u_1,u_2)\rightarrow (v_1,v_2)$ in ${\bf H}$, then for $\psi_1, \varphi_2\in H^1(\R^3)$,
\begin{align*}
&\int_{\R^3}\int_0^{2\pi}(u_2+u_1 \cos \theta )(u_1+u_2\cos\theta)\Big(u_1^2+2u_1 u_2 \cos \theta+u_2^2\Big)^\frac{p-3}{2}\\
&\quad -(v_2+v_1 \cos \theta )(v_1+v_2\cos\theta)\Big(v_1^2+2v_1 v_2 \cos \theta+v_2^2\Big)^\frac{p-3}{2}d\theta \psi_1\varphi_2 dx\\
&\le \|\psi_1\|_{L^{p+1}}\|\varphi_2\|_{L^{p+1}}\bigg\|\int_0^{2\pi}(u_2+u_1 \cos \theta )(u_1+u_2\cos\theta)\Big(u_1^2+2u_1 u_2 \cos \theta+u_2^2\Big)^\frac{p-3}{2}\\
&\quad -(v_2+v_1 \cos \theta )(v_1+v_2\cos\theta)\Big(v_1^2+2v_1 v_2 \cos \theta+v_2^2\Big)^\frac{p-3}{2}d\theta\bigg\|_{L^\frac{p+1}{p-1}}\rightarrow 0.
\end{align*}
Thus we deduce that $A$ is $C^2$, with derivatives given by
\begin{align*}
  A^\prime(\vec{u})\vec{\psi}&=\int_{\R^3}\int_0^{2\pi}\Big(u_1^2+2u_1 u_2 \cos \theta+u_2^2\Big)^\frac{p-1}{2} \Big((u_1+u_2 \cos \theta )\psi_1   +(u_2+u_1 \cos \theta ) \psi_2\Big) d\theta dx
\end{align*}
and
\begin{equation}\label{ac2e}
\begin{aligned}
&  A^{\prime \prime}(\vec{u})[\vec{\psi},\vec{\psi}]\\
&=\int_{\R^3}\int_0^{2\pi} \bigg[\Big(u_1^2+2u_1 u_2 \cos \theta+u_2^2\Big)^\frac{p-1}{2} +(p-1)(u_1+u_2 \cos \theta )^2\Big(u_1^2+2u_1 u_2 \cos \theta+u_2^2\Big)^\frac{p-3}{2}\bigg]\psi_1^2 \\
&\quad +2\bigg[  \cos \theta  \Big(u_2^2+2u_2 u_1 \cos \theta+u_1^2\Big)^\frac{p-1}{2}\\
&\quad +(p-1)(u_2+u_1 \cos \theta )(u_1+u_2\cos\theta)\Big(u_2^2+2u_2 u_1 \cos \theta+u_1^2\Big)^\frac{p-3}{2}\bigg]\psi_1 \psi_2\\
&\quad +\bigg[ \Big(u_2^2+2u_2 u_1 \cos \theta+u_1^2\Big)^\frac{p-1}{2} +(p-1)(u_2+u_1 \cos \theta )^2\Big(u_2^2+2u_2 u_1 \cos \theta+u_1^2\Big)^\frac{p-3}{2}\bigg] \psi_2^2  d\theta dx\\
&=\int_{\R^3}\int_0^{2\pi} \bigg\{\Big[1+(p-1)(u_1+u_2 \cos \theta )^2 (u_1^2+2u_1 u_2 \cos \theta+u_2^2 )^{-1}\Big]\psi_1^2 \\
&\quad +2\Big[  \cos \theta  +(p-1)(u_2+u_1 \cos \theta )(u_1+u_2\cos\theta) (u_2^2+2u_2 u_1 \cos \theta+u_1^2 )^{-1}\Big]\psi_1 \psi_2\\
&\quad +\Big[ 1 +(p-1)(u_2+u_1 \cos \theta )^2 (u_1^2+2u_1 u_2 \cos \theta+u_2^2 )^{-1}\Big] \psi_2^2 \bigg\}\Big(u_2^2+2u_2 u_1 \cos \theta+u_1^2\Big)^\frac{p-1}{2} d\theta dx.
\end{aligned} 
\end{equation} 
Thus, from   \eqref{esee2} and \eqref{ac2e},   we have
\begin{equation}\label{hss1}
\begin{aligned}
&I^{\prime \prime}(\vec{u})[\vec{\psi},\vec{\psi}]\\
&= \int_{\R^3} |\nabla \psi_1|^2 +  |\nabla \psi_2|^2+\lambda  \psi_1^2+\lambda  \psi_2^2dx\\
&\quad+ \int_{\R^3}\mu_{11}   \phi_{u_1} \psi_1^2+\mu_{11}   u_1 (\phi_{u_1}^\prime \psi_1) \psi_1- \mu_{12}   \phi_{u_2} \psi_1^2 -2\mu_{12}  u_2 (\phi_{u_1}^\prime \psi_1) \psi_2 \\
&\quad  + \mu_{22}     \phi_{u_2} \psi_2^2 + \mu_{22}    u_2  (\phi_{u_2}^\prime \psi_2) \psi_2   -\mu_{12}\phi_{u_1}\psi_2^2dx\\
&\quad -  \frac{1}{2\pi}\int_{\R^3}\int_0^{2\pi} \bigg\{\Big[1+(p-1)(u_1+u_2 \cos \theta )^2 (u_1^2+2u_1 u_2 \cos \theta+u_2^2 )^{-1}\Big]\psi_1^2 \\
&\quad +2\Big[  \cos \theta  +(p-1)(u_2+u_1 \cos \theta )(u_1+u_2\cos\theta) (u_2^2+2u_2 u_1 \cos \theta+u_1^2 )^{-1}\Big]\psi_1 \psi_2\\
&\quad +\Big[ 1 +(p-1)(u_2+u_1 \cos \theta )^2 (u_1^2+2u_1 u_2 \cos \theta+u_2^2 )^{-1}\Big] \psi_2^2 \bigg\}\Big(u_2^2+2u_2 u_1 \cos \theta+u_1^2\Big)^\frac{p-1}{2} d\theta dx.
\end{aligned}
\end{equation}
\end{proof}

\noindent{\bf Acknowledgement.}
This research of the first author was supported by the Basic Science Research Program through the National Research Foundation of Korea (NRF) funded by the Ministry of Science and ICT (NRF-2020R1A2C4002615).
This research of the second author was supported by Basic Science Research Program through the National Research Foundation of Korea (NRF) funded by the Ministry of Science and ICT (NRF-2020R1C1C1A01006415).


\begin{thebibliography}{00}
\bibitem{AA}
N. Akhmediev, A. Ankiewicz, Partially Coherent Solitons on a Finite
Background, Physical Review Letters, {\bf 82} (1999), 2661–2664.


\bibitem{AC}
A. Ambrosetti and E. Colorado, Standing waves of some coupled nonlinear Schr\"odinger equations, J. Lond. Math. Soc. {\bf 75} (2007), 67–82.

\bibitem{AR} A. Ambrosetti, P. H. Rabinowitz,   Dual variational methods in critical point theory and applications. J. Functional Analysis {\bf 14} (1973), 349–381.

\bibitem{AP}
A. Azzollini, A. Pomponio, Ground state solutions for the nonlinear Schr\"odinger-Maxwell equations. J. Math. Anal. Appl. {\bf 345} (2008), no. 1, 90–108.

\bibitem{BDW}
T. Bartsch, E. N. Dancer and Z.-Q. Wang, A Liouville theorem, a-priori bounds, and bifurcating branches of positive solutions for a nonlinear elliptic system, Calc. Var. Partial Differential Equations {\bf 37} (2010) 345–361.


\bibitem{BL}
H. Berestycki, P. L. Lions,   Nonlinear scalar field equations  I. Existence of a ground state. Arch. Rational Mech. Anal. {\bf 82} (1983), no. 4, 313–345.

\bibitem{C}
G. M. Coclite,   A multiplicity result for the nonlinear Schr\"odinger-Maxwell equations, Commun. Appl. Anal. {\bf7} (2003), no. 2-3, 417–423.


\bibitem{CW}
M. Colin, T. Watanabe,  Standing waves for the nonlinear Schr\"odinger equation coupled with the Maxwell equation, Nonlinearity {\bf30} (2017), no. 5, 1920–1947.


\bibitem{DM1}
T. D'Aprile, D. Mugnai,   Solitary waves for nonlinear Klein-Gordon-Maxwell and Schr\"odinger-Maxwell equations, Proc. Roy. Soc. Edinburgh Sect. A {\bf134} (2004), no. 5, 893–906.

\bibitem{DM}
T. D'Aprile, D. Mugnai, Non-existence results for the coupled Klein-Gordon-Maxwell equations. Adv. Nonlinear Stud.{\bf 4 }(2004), no. 3, 307–322.

\bibitem{E}
I. Ekeland,  On the variational principle. J. Math. Anal. Appl. {\bf 47} (1974), 324–353.

\bibitem{GT}
D. Gilbarg and N. Trudinger, Elliptic Partial Differential Equations of Second Order, {Springer-Verlag}, 1983.


\bibitem{H}
H. Hofer,  A note on the topological degree at a critical point of mountainpass-type, Proc. Amer. Math. Soc. {\bf 90} (1984), no. 2, 309–315.

\bibitem{JT}
L. Jeanjean, K. Tanaka,  A remark on least energy solutions in $\R^N$. Proc. Amer. Math. Soc. {\bf 131} (2003), no. 8, 2399–2408. 


\bibitem{1JS}
W. Jeong, J. Seok, On perturbation of a functional with the mountain pass geometry: applications to the nonlinear Schrödinger-Poisson equations and the nonlinear Klein-Gordon-Maxwell equations. Calc. Var. Partial Differential Equations {\bf 49} (2014), no. 1-2, 649–668.

\bibitem{L}
E. H. Lieb,   Existence and uniqueness of the minimizing solution of Choquard's nonlinear equation. Studies in Appl. Math. {\bf 57} (1976/77), no. 2, 93–105.

\bibitem{LL}
E.H. Lieb, M. Loss,  Analysis. Second edition. Graduate Studies in Mathematics,  14. American Mathematical Society, Providence, RI, 2001.


\bibitem{Li}
P. L. Lions, The Choquard equation and related questions, Nonlinear Anal. {\bf4} (1980), no. 6, 1063–1072. 


\bibitem{Le}
E. Lenzmann,  Uniqueness of ground states for pseudorelativistic Hartree equations. Anal. PDE {\bf 2} (2009), no. 1, 1–27. 

\bibitem{LW}
Z. Liu and Z.-Q. Wang, Ground states and bound states of a nonlinear Schrdinger system, Adv. Nonlinear Stud. {\bf 10} (2010), 175–193.

\bibitem{MN}
N. Masmoudi, K. Nakanishi,  From nonlinear Klein-Gordon equation to a system of coupled nonlinear Schr\"odinger equations, Math. Ann. {\bf 324} (2002), no. 2, 359–389.

\bibitem{MN1}
N. Masmoudi, K. Nakanishi,   Nonrelativistic limit from Maxwell-Klein-Gordon and Maxwell-Dirac to Poisson-Schr\"odinger, Int. Math. Res. Not. 2003, no.  13, 697–734.

\bibitem{MRS}
P.A. Markowich, C. A. Ringhofer, C. Schmeiser,   Semiconductor equations, Springer-Verlag, Vienna, 1990. 


\bibitem{MZ}
L. Ma, L. Zhao,  Classification of positive solitary solutions of the nonlinear Choquard equation. Arch. Ration. Mech. Anal. {\bf195} (2010), no. 2, 455–467. 

\bibitem{MCSS}
M. Mitchell, Z. Chen, M. Shih and M. Segev, Self-Trapping of partially
spatially incoherent light, Phys. Rev. Lett. {\bf 77 } (1996), 490–493.

\bibitem{MS}
M. Mitchell and M. Segev, Self-trapping of inconherent white light, Nature, {\bf 387} (1997), 880–882.

\bibitem{PW}
S. Peng, Z.Q. Wang,  Segregated and synchronized vector solutions for nonlinear Schrödinger systems. Arch. Ration. Mech. Anal. {\bf 208} (2013), no. 1, 305–339.


\bibitem{RCF}
CH. Regg, N. Cavadini, A. Furrer, et al, Bose-Einstein condensation of the triple states in the magnetic insulator $TlCuCl_3$, Nature, {\bf 423} (2003), 62–65.

\bibitem{R}
D. Ruiz, The Schrödinger-Poisson equation under the effect of a nonlinear local term. J. Funct. Anal. {\bf 237 } (2006), no. 2, 655–674.

\bibitem{S1}
B. Sirakov, Least energy solitary waves for a system of nonlinear Schrödinger equations in $R^n$, Comm. Math. Phys. {\bf 271} (2007), 199–221.

\bibitem{S}
W. A. Strauss,   Existence of solitary waves in higher dimensions. Comm. Math. Phys. {\bf 55} (1977), no. 2, 149–162.

\bibitem{TBDC}
G. Thalhammer, G. Barontini, L. De Sarlo L, J. Catani, F. Minardi and M. Inguscio, Double species Bose-Einstein condensate with tunable
interspecies interactions, Phys Rev Lett. {\bf 100} (2008), 210402.


\bibitem{V}
G. Vaira,  Ground states for Schr\"odinger-Poisson type systems, Ric. Mat. {\bf 60} (2011), no. 2, 263–297. 


\bibitem{WS}
J. Wang, J. Shi,   Standing waves for a coupled nonlinear Hartree equations with nonlocal interaction. Calc. Var. Partial Differential Equations  {\bf 56} (2017), no. 6, Paper No. 168.
\bibitem{WY}
 J. Wei and W. Yao, Uniqueness of positive solutions to some coupled
nonlinear Schr\"odinger equations, Commun. Pure Appl. Anal. {\bf 11} (2012),
1003–1011.
\end{thebibliography}
\end{document}